\documentclass[a4paper,12pt]{amsart}

\usepackage{euscript,eufrak,verbatim}
\usepackage{graphicx}
\usepackage[usenames]{color}
\usepackage[colorlinks,linkcolor=red,anchorcolor=blue,citecolor=blue]{hyperref}
\usepackage[psamsfonts]{amssymb}

\newtheorem{thm}{Theorem}[section]
\newtheorem{prop}[thm]{Proposition}
\newtheorem{lem}[thm]{Lemma}
\newtheorem{cor}[thm]{Corollary}

\def\dimLow{\underline{\rm dim}}

\def\Zp{\mathbb{Z}_{p}}
\def\Qp{\mathbb{Q}_{p}}

\def\Q{\mathbb{Q}}
\def\Z{\mathbb{Z}}
\def\ZE{\mathcal{Z}_{\widehat{\delta_E}}}

\def\R{\mathbb{R}}
\def\M{\mathcal{M}}

\begin{document}
	
\title[On $p$-adic spectral measures]{On $p$-adic spectral measures}

\author
{Ruxi Shi}
\address
{Institute of Mathematics, Polish Academy of Sciences, ul. \'Sniadeckich 8, 00-656 Warszawa, Poland}
\email{rshi@impan.pl}



\begin{abstract}
A Borel probability measure $\mu$ on a locally compact group is called a spectral measure if there exists a subset of continuous group characters which forms
an orthogonal basis of the Hilbert space $L^2(\mu)$. In this paper, we characterize all spectral measures in the field $\Qp$ of $p$-adic numbers. 
\end{abstract}

\maketitle


\section{Introduction}

Let $G$ be a locally compact abelian group and $\widehat{G}$  its dual group. We say that a Borel measure $\mu$  on $G$  is a {\em spectral measure} if there exists a set $\Lambda \subset \widehat{G}$ which is
an orthonormal basis of the Hilbert space $L^2(\mu)$. Such a set $\Lambda$ is called a {\em spectrum} of $\mu$ and the pair $(\mu,\Lambda)$ is called a {\em spectral pair}.

\medskip

The study of spectral measures was pioneered by Fuglede \cite{F}. He formulated so-called spectral set conjecture which asserted that 

\medskip

{\em 
The measure $1_\Omega dx$ is a spectral measure in $G$ if and only if the set $\Omega$  is a tile of $G$ by translation.}

\medskip

\noindent 
Here, the Haar measure on $G$ is denoted by $dx$, and the set $\Omega$ is called a \textit{spectral set} if the Haar measure restricted on $\Omega$, which is denoted by $1_\Omega dx$, is a spectral measure.
Although the conjecture was disproved eventually for $\R^d$ with $d\ge 3$ \cite{TT, M, KM, KM2, FG, FMM06}, it remains still widely open for general locally compact abelian groups, especially for abelian groups in lower ``dimension". We only know that Fuglede's spectral set conjecture holds on some specific groups, for example, $\mathbb{Z}_{p^n}$ \cite{l2}, $\mathbb{Z}_p \times \mathbb{Z}_p$ \cite{imp}, $\mathbb{Z}_{p^nq}$ with $n\ge 1$ \cite{mk}, $p$-adic field $\mathbb{Q}_p$ \cite{FFS, FFLS}, $\Z_{pqr}$ with $p,q,r$ different primes  \cite{Shi1}, very recently $\mathbb{Z}_{p^nq^2}$ with $n\ge 1$ \cite{kmsv} and $\Z_{p^2}\times \Z_{p}$ \cite{Shi2}.
  
\medskip

Jorgensen and Pedersen \cite{jp} discovered that the standard middle-fourth Cantor measure is a spectral measure, which is the first spectral measure that is non-atomic and singular to the Haar measure (on $\R$) ever discovered. In the same paper, they also showed that the middle-third Cantor measure is not a spectral measure. After them, it becomes an active research area on determining self-similar/self-affine/Cantor-Moran spectral measures (see for example \cite{Shi4, Aks}).

\medskip

Throughout the paper, we consider the case $G=\Qp$ the $p$-adic field. In what follows, it is always with respect to Haar measure $dx$ whenever we say that a Borel measure in $\Qp$ is singular (respectively absolutely continuous). 
Fan et al. \cite{FFLS} gave a geometrical description of spectral sets in $\Qp$, where they proved that a Borel set is a spectral set in $\Qp$ if and only if it is a tile by translation, and if and only if its centers form a $p$-homogeneous tree. Moreover, they constructed examples of singular spectral measures which are the weak limits of  absolutely continuous measures  whose density functions are certain indicator functions of spectral sets in $\Qp$.  We recall their construction as follows:

\medskip

Let $\mathbb{I}, \mathbb{J}$ be two disjoint  infinite subsets of $\mathbb{N}$ such that
$$\mathbb{I}\bigsqcup \mathbb{J}=\mathbb{N}.$$
For any non-negative integer $\gamma$, let $\mathbb{I}_{\gamma}= \mathbb{I} \cap  \{0,1, \cdots \gamma-1 \}$ and $\mathbb{J}_{\gamma}=\mathbb{J} \cap  \{0,1, \cdots \gamma-1 \}$. Let $C_{\mathbb{I}_\gamma,\mathbb{J}_{\gamma}}\subset \Z/p^{\gamma}\Z$ be $p$-homogeneous  subsets having $(0, \gamma, \mathbb{I}_\gamma, \mathbb{J}_\gamma )$-tree structure (roughly speaking, it means that the digit set of $p$-expansions of elements in $C_{\mathbb{I}_\gamma,\mathbb{J}_{\gamma}}$ is homogenous according to the sets $\mathbb{I}_\gamma$ and $\mathbb{J}_{\gamma}$, see the precise definitions in Section \ref{sec:tree structure}).
Considering $C_{\mathbb{I}_\gamma, \mathbb{J}_{\gamma}}$ as a subset of $\Z_p$, let
$$
\Omega_{\gamma}=\bigsqcup_{c \in C_{I_\gamma, J_\gamma}} \left(c + p^\gamma \mathbb{Z}_p\right), \quad \gamma=0,1,2\cdots 
$$
be a nested sequence of compact open sets, i.e.  $\Omega_{0}\supset \Omega_{1}\supset \Omega_{2}\supset \cdots$.
It is obvious that the measures ${\frac{1}{|\Omega_{\gamma}|}1_{\Omega_{\gamma}}dx}$ (and also $\frac{1}{\sharp \mathbb{I}_{\gamma}}\sum_{c \in C_{I_\gamma, J_\gamma}} \delta_c$) weakly converge to a singular measure, namely $\nu_{\mathbb{I},\mathbb{J}}$, as $\gamma \to \infty$. The measure $\nu_{\mathbb{I},\mathbb{J}}$ is supported on a $p$-homogeneous, Cantor-like set of measure $0$, and the measure of an open ball with respect to  $\nu_{\mathbb{I},\mathbb{J}}$  is just the ``proportion" of this support inside the ball.

\medskip

Fan et al. \cite{FFLS} proved that the singular measure $\nu_{\mathbb{I},\mathbb{J}}$ under construction is a spectral measure. We remark that in the above construction if we set $\mathbb{I}$ to be finite (respectively $\mathbb{J}$ finite) then the associated measure $\nu_{\mathbb{I},\mathbb{J}}$ is discrete (respectively absolutely continuous). 

\medskip

It is easy to check that a spectral measure under translation or multiplier is still spectral. Thus the translation or multiplier of $\nu_{\mathbb{I},\mathbb{J}}$ is also a spectral measure. Since the measure  $\nu_{\mathbb{I},\mathbb{J}}$ is constructed in a simple and intuitive way, it seems that the spectral measures in $\Qp$ should involve measures with more sophisticated structures. But we disestablish such semblance by showing the rigidity of the spectral measures. More precisely, we prove that  the measures $\nu_{\mathbb{I},\mathbb{J}}$ (with $\mathbb{I}$ and $\mathbb{J}$ a partition of $\mathbb{N}$) which are constructed above are all spectral measures in $\Qp$ up to translation or  multiplier. Now we state our main theorem.

\begin{thm}\label{thm:main}
	A probability measure $\mu$ in $\Qp$ is a spectral measure if and only if there exist two sets $\mathbb{I}$ and $ \mathbb{J}$ that form a partition of $\mathbb{N}$ such that the measure $\mu$ is of the form $\nu_{\mathbb{I},\mathbb{J}}$ up to translation and multiplier.
\end{thm}


As a consequence of Theorem \ref{thm:main}, we could calculate the precise value of dimensions of spectral measures and their spectra, and establish certain equality between them. We remark that one could not expect that the equality holds for spectral measures in general locally compact abelian groups. We refer to \cite{Shi3} for general cases.
\begin{prop}\label{cor:dimension}
		Let $\mu$ be a spectral measure in $\Qp$ with spectrum $\Lambda$. Then we have
		$$
		\dim_H \mu=\underline{\dim}_e \mu~\text{and}~
		\dim_B \Lambda=\overline{\dim}_{\text{e}} \mu.
		$$
\end{prop}

We end up this section with presenting the structure of the paper and the rough idea of the proof of Theorem \ref{thm:main}.
\subsection*{Very rough idea of proof}
We firstly see that the functional equation 
\begin{equation}\label{eq:idea 1}
|\widehat{\mu}|^2 * \delta_{\Lambda} = 1.
\end{equation}
is a criteria for spectral probability measure $\mu$ of $\Qp$ with spectrum $\Lambda$. By taking Fourier transformation of \eqref{eq:idea 1}, we have that
\begin{equation}\label{eq:idea 2}
\widehat{|\widehat{\mu}|^2} \cdot \widehat{\delta}_{\Lambda} = \delta_0.
\end{equation}
Then we show that the union of ``zeros" of $\widehat{|\widehat{\mu}|^2}$ and $\widehat{\delta}_{\Lambda}$ is equal to the set of ``zeros" of $ \delta_0$ (see precise definition of ``zeros" of distributions in Section \ref{subset:zeros}). In fact, using Colombeau algebra of generalized functions in $\Qp$, we prove such property for general functional equation $f\cdot h=g$ for $f,g,h$ distributions under a mild condition (see Proposition \ref{prop:union of zeros}). Since the distribution $\delta_0$ has abundant ``zeros", we finally discover the structure of $\mu$ and $\Lambda$ with the help of ``zeros" of $\widehat{|\widehat{\mu}|^2}$ and $\widehat{\delta}_{\Lambda}$. Actually, we use $p^n$-cycles (see Section \ref{Zmodule}) and $p$-homogeneous set (see Section \ref{sec:tree structure}) as tools to investigate the local structures of $\mu$ and $\Lambda$. More precisely, we will firstly show that $\Lambda$ in a very small ball contains a relatively large $p$-homogeneous set and the support of $\mu$ in a very small ball is contained in a  relatively small $p$-homogeneous set.  Then we carefully enlarger the ball which we look at and show that these two $p$-homogeneous sets are of the exactly same size (here, the rigidity of spectral measure is applied). Finally, we continue this process until that we get the structures of $\mu$ and $\Lambda$ on the whole $\Qp$.
 We remark that the primeness of $p$ plays an important role not only in studying the functional equation \eqref{eq:idea 2}, but also investigating the structures of $\mu$ and $\Lambda$.

\subsection*{Structure of the paper} In Section \ref{sec:disctribution}, we recall the theory of Bruhat-Schwartz distributions and Colombeau algebra of generalized functions in $\Q_p$. In Section \ref{sec:dimension}, we recall basic notions and properties of dimensions of sets and measures. In Section \ref{sec:tree structure}, we introduce the notion of $p$-homogenous sets in $\Qp$. In Section \ref{sec:zeros}, we prove the proposition about the zeros of distributions and that of their product, which is crucial to the proof of Theorem \ref{thm:main}. In Section \ref{Zmodule}, we study the vanishing sum of continuous characters, where the $p$-homogenous sets play an important role. In Section \ref{sec:density}, we discuss the density of uniformly discrete set whenever the set satisfies a simply functional equation. In Section \ref{sec: spectral measures in Qp}, we investigate the properties of spectral measures and their spectra, and then prove Theorem \ref{thm:main}. In Section \ref{sec:cor dimension}, we prove Proposition \ref{cor:dimension}. In Section \ref{sec:higher dimension}, we discuss several properties of spectral measures in higher dimensional $p$-adic spaces.

\setcounter{equation}{0}

\section{Distribution and generalized function on $\Qp$}\label{sec:disctribution}

\subsection{The field of $p$-adic number $\Qp$}\label{p-adicfield}

We begin with a quick recall of the field of $p$-adic numbers. 
Consider the field $\mathbb{Q}$ of rational numbers and a prime $p\ge 2$.
Any nonzero number $r\in \mathbb{Q}$ can be written as
$r =p^v \frac{a}{b}$ where $v, a, b\in \mathbb{Z}$ and $(p, a)=1$ and $(p, b)=1$
where $(x, y)$ denotes the greatest common divisor of the two integers $x$ and $y$. 
 We define the non-Archimedean absolute value
$|r|_p = p^{-v_p(r)}$ for $r\not=0$ and $|0|_p=0$. That means\\
\indent (i)  \ \ $|r|_p\ge 0$ with equality only when $r=0$; \\
\indent (ii) \ $|r s|_p=|r|_p |s|_p$;\\
\indent (iii) $|r+s|_p\le \max\{ |r|_p, |s|_p\}$.\\
The field $\mathbb{Q}_p$ of $p$-adic numbers is defined as the completion of $\mathbb{Q}$ under
$|\cdot|_p$. In other words, a typical element $x$ of $\mathbb{Q}_p$ is of the form
\begin{equation}\label{HenselExp}
x= \sum_{n= v}^\infty a_n p^{n} \qquad (v\in \mathbb{Z}, a_n \in \{0,1,\cdots, p-1\} \text{ and } a_v\neq 0). 
\end{equation}
Here, $v_p(x):=v$ is called the $p$-{\em valuation} of $x$. The ring $\mathbb{Z}_p$ of $p$-adic integers is the set of $p$-adic numbers with absolute value smaller than or equal to $1$. 

 A non-trivial additive character on $\mathbb{Q}_p$ is defined by the formula
$$
\chi(x) = e^{2\pi i \{x\}}
$$
where $\{x\}= \sum_{n=v_p(x)}^{-1} a_n p^n$ is the fractional part of $x$ in (\ref{HenselExp}). From this character we can get all characters $\chi_\xi$ of $\mathbb{Q}_p$, by defining 
$\chi_\xi(x) =\chi(\xi x)$. It is not hard to see that
\begin{align}\label{one-in-unit-ball}
\chi(x)=e^{2\pi i k/p^n}, \quad  \text{  if }  x \in \frac{k}{p^n}+\mathbb{Z}_p  \ \  (k, n \in \Z), 
\end{align}
and
\begin{align}\label{integral-chi}
\int_{p^{-n}\mathbb{Z}_p} \chi(x)dx=0 \ \text{ for all } n\geq 1.
\end{align}
In fact, the map $\xi \mapsto \chi_{\xi}$ from $\Q_p$ to $\widehat{\Q}_p$ is an isomorphism. We thus write $\widehat{\Q}_p\simeq \Qp$ and identify a point $\xi \in \Q_p$ with the point $\chi_\xi \in \widehat{\Q}_p$.  
For more information on $\mathbb{Q}_p$ and $\widehat{\Q}_p$, the reader is referred to the book \cite{Vvz}.

The following notation will be used for convenience in the whole paper.

\medskip
\begin{tabular}{|c|c|}
	\hline $\mathbb{Z}_p^\times $ & $\mathbb{Z}_p\setminus p\mathbb{Z}_p=\{x\in \mathbb{Q}_p: |x|_p=1\}$,
	the group of units of $\mathbb{Z}_p$\\
	\hline $B(0, p^{n})$ & $p^{-n} \mathbb{Z}_p$,  the (closed) ball centered at $0$ of radius $p^n$\\
	\hline  $B(x, p^{n})$ &$x + B(0, p^{n})$\\
	\hline $ S(x, p^{n})$ & $B(x, p^{n})\setminus B(x, p^{n-1})$,  the sphere centered at $0$ of radius $p^n$\\
	\hline $ \mathbb{L} $ & $\{\{x\}: x\in \Q_p\}$, a complete set of representatives \\
	& of the cosets of the additive subgroup $\mathbb{Z}_p$\\
	\hline $ \mathbb{L}_n$ & $ p^{-n} \mathbb{L}$\\
	\hline
\end{tabular}

\medskip

Finally, when considering the geometrical structure of $\Qp$ and its subsets, we often keep in mind two models of $\Qp$: tree model and ball model. 
See Figure \ref{Fig:ball model} and Figure \ref{Fig:tree model}.

\begin{figure}[h!]
	\centering
	\includegraphics[width=0.4\textwidth]{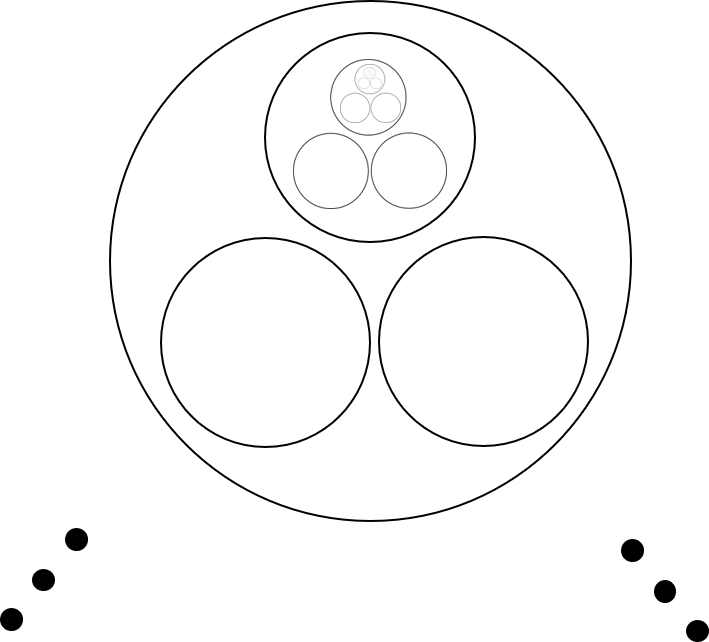}\\
	\caption{Ball model of $\Q_3$.} 
	\label{Fig:ball model}
\end{figure}

\begin{figure}[h!]
	\centering
	\includegraphics[width=0.8\textwidth]{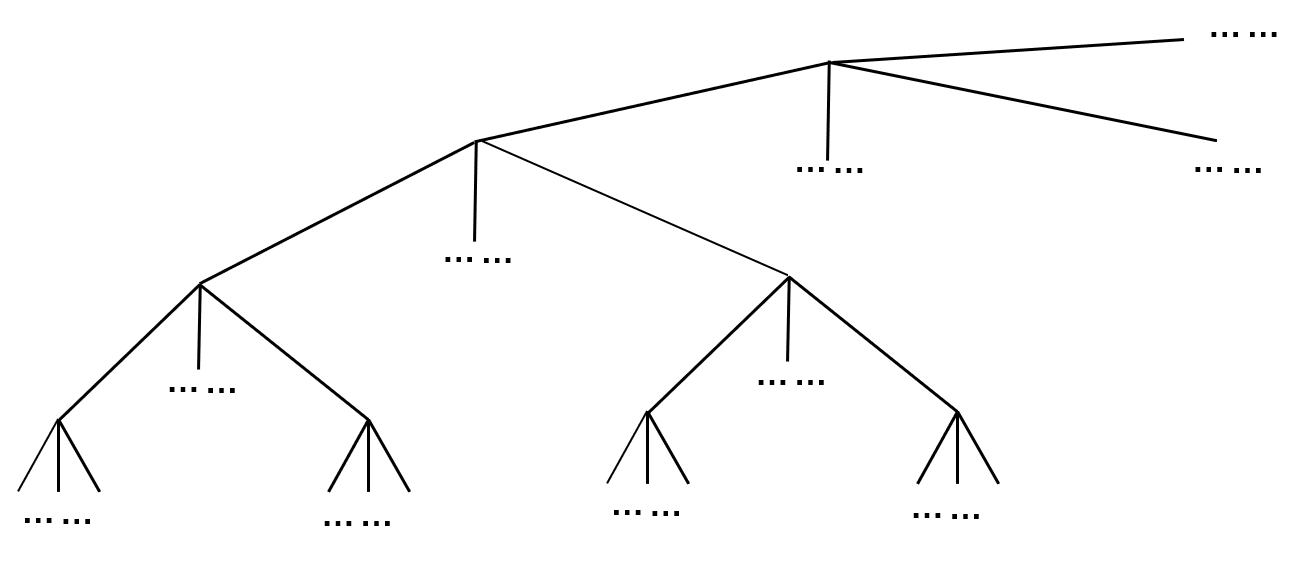}\\
	\caption{Tree model of $\Q_3$.} 
	\label{Fig:tree model}
\end{figure}







\medskip

\subsection{Fourier transformation of $L^1$ functions}
The Fourier transformation of  $f\in L^1(\Q_p)$ is defined to be 
$$\widehat{f}(\xi)=\int_{\Qp}f(x)\overline{\chi_\xi(x)} dx   \quad (\forall \xi\in \widehat{\Q}_p\simeq \Qp).$$
A complex function $f$ defined on $\Q_p$ is called \textit{uniformly locally constant} if  there exists $n\in \mathbb{Z}$ such that 
\[f(x+u)=f(x) \quad \forall x\in \Q_p, \forall u \in B(0, p^n).\] 


The following proposition shows that for  an integrable function $f$, having compact support and being uniformly locally constant are dual properties for $f$ and its Fourier transform.
 
\begin{prop}[\cite{FFLS}, Proposition 2.2]\label{2case}
	Let $f\in L^1(\Q_p)$ be  a complex-value integrable function. \\
	\indent {\rm (1)} If $f$ has compact support, then $\widehat{f}$ is uniformly locally constant.\\
   \indent {\rm (2)} If $f$ is uniformly locally constant, then $\widehat{f}$ has compact support.
\end{prop}

\subsection{Convolution and Fourier transformation of measures}
Let $\M(\Qp)$ be the set of all regular bounded measures on $\Qp$. It is clear that $\M(\Qp)$ is a normed linear space. Let $\nu, \mu\in \M(\Qp)$. Let $\nu\times \mu$ be their product measure on the product space $\Qp\times \Qp$. For a Borel set $E\subset \Qp$, let $E_+$ be the set
$$
\{(x,y)\in \Qp\times \Qp: x+y\in E \}.
$$
Then the convolution of $\nu$ and $\mu$, denoted by $\nu*\mu$, is defined by
$$
(\nu*\mu)(E)=(\nu\times \mu)(E_+),
$$
for any Borel set $E$ in $\Qp$.
It is well known that $\nu*\mu\in \M(\Qp)$ and that  if $\nu$ and $\mu$ are probability measures, then so is $\nu*\mu$.

The Fourier transformation of  $\mu\in \M(\Qp)$ is defined to be 
$$\widehat{\mu}(\xi)=\int_{\Qp}\overline{\chi_\xi(x)} d\mu(x)   \quad (\forall \xi\in \widehat{\Q}_p\simeq \Qp).$$

\subsection{Bruhat-Schwartz distributions in $\Q_p$}\label{subsec2.4}
Here we give a brief description of the theory of Bruhat-Schwartz distributions in $\Q_p$ which mainly follows the content in
\cite{Aks,t,Vvz}. 
  Let $\mathcal{E}$ denote the space of the uniformly locally constant functions.
The space $\mathcal{D}$ of {\em Bruhat-Schwartz test functions} is, by definition, constituted of uniformly locally constant functions
 with compact support. In fact, such a test function $f\in \mathcal{D}$ is a finite linear combination of indicator functions of the form $1_{B(x,p^k)}(\cdot)$, where $k\in \mathbb{Z}$ and $x\in \Q_p$. The largest of such numbers $k$, denoted by $\ell:= \ell(f)$, is  called the {\em parameter of constancy} of $f$. Since $f\in \mathcal{D}$ has compact support, the minimal number $\ell':=\ell'(f)$ such that the support of $f$ is contained in $B(0, p^{\ell'})$ exists and is called the {\em parameter of compactness} of $f$.

 Clearly, we have the relation $\mathcal{D}\subset \mathcal{E}$. The space $\mathcal{D}$ is equipped with the topology as follows: a sequence $\{\phi_n \}\subset \mathcal{D}$ is called a {\em null sequence} if there is a fixed pair of $l, l^{\prime}\in \mathbb{Z}$ such that 
 \begin{itemize}
 	\item each $\phi_n$ is constant on every ball of radius $p^l$;
 	\item each $\phi_n$ is supported by the ball $B(0,p^{l^{\prime}})$;
 	\item the sequence $\phi_n$ tends uniformly to zero.
 \end{itemize}

 A {\em Bruhat-Schwartz distribution} $f$ on $\Q_p$ is by definition a continuous linear functional on $\mathcal{D}$. The value of $f$ at $\phi \in \mathcal{D}$
 will  be denoted by $\langle f, \phi \rangle$.  
 Note that linear functionals on $\mathcal{D}$ are automatically continuous. This property allows us to easily construct distributions. Denote by
 $\mathcal{D}'$ the space of  Bruhat-Schwartz  distributions. The space $\mathcal{D}'$ is provided with the weak topology induced by  $\mathcal{D}$.
 
  A locally integrable function $f$ is considered as a distribution:  for any $\phi \in \mathcal{D}$,
$$
\langle f,\phi\rangle=\int_{\Q_p} f\phi dx.
$$
A subset $E$ of $\Q_p$ is said to be {\em uniformly discrete} if $E$ is countable and $\inf_{x,y \in E} |x-y|_p>0$. Remark that if $E$ is uniformly discrete, then  ${\rm Card}(E\cap K)<\infty$  for any compact subset $K$ of $\Q_p$ so that  
\begin{equation}\label{mu_E}
\delta_E = \sum_{\lambda\in E} \delta_\lambda
\end{equation}
defines a  discrete
Radon measure, which is also a distribution:
for any $\phi \in \mathcal{D}$,
$$
\langle \delta_E,\phi\rangle= \sum_{\lambda\in E} \phi(\lambda).
$$
Here for each $\phi$, the sum is finite  because $E$ is uniformly discrete and thus each ball contains at most a finite number of points in $E$.
Since the test functions in $\mathcal{D}$ are  uniformly locally constant  and have compact support,
the following proposition is a direct consequence of the fact (see for example \cite[Lemma 4]{Fan}) that
\begin{equation}\label{FB}
\widehat{1_{B(c, p^k)}}(\xi) = \chi(-c \xi) p^{k} 1_{B(0, p^{-k})}(\xi).
\end{equation}

\begin{prop}[\cite{t}, Chapter II 3]
 The Fourier transformation  $f \mapsto \widehat{f}$ is a  homeomorphism  from $\mathcal{D}$ onto $\mathcal{D}$.
  \end{prop}

 The {\em Fourier transform of a distribution} $ f\in \mathcal{D}'$ is a new distribution  $ \widehat{f}\in \mathcal{D}'$ defined by the duality
$$
\langle\widehat{f},\phi\rangle=\langle f,\widehat{\phi}\rangle, \quad \forall \phi \in \mathcal{D}.
$$
Actually, the Fourier transformation $f\mapsto \widehat{f}$ is a homeomorphism of $\mathcal{D}'$ onto $\mathcal{D}'$ under the weak topology \cite[ Chapter II 3]{t}.

\subsection{Zeros of distribution and Fourier transform  of a discrete  measure}\label{subset:zeros}

Let $f \in \mathcal{D}'$ be a distribution in $\Q_p$.  A  point $x\in \Q_p$ is  called a {\em  zero} of  $f$ if there exists an integer $n_0$
  such that  $$ \langle f, 1_{B(y,p^{n})}\rangle=0, \quad \text {for all }  y\in B(x,p^{n_0})  \text{ and all integers  }  n\leq  n_0 \text.$$
Hereafter, it will be convenient to use  $\mathcal{Z}_f$ to denote the set consisting of all zeros of $f$.
Observe that 
 $\mathcal{Z}_f$ is  the maximal open set $U$ on which $f$ vanishes, i.e.
 $\langle f, \phi\rangle=0$ for all $\phi \in \mathcal{D}$ such that the support of $\phi$ is contained in $U$.
 The {\em support}  of a distribution $f$ is defined as  the complementary set of  $\mathcal{Z}_f$ and is denoted by  $\operatorname{supp}(f)$.

Let $E$ be a  uniformly discrete set in $\Q_p$. 
Define the quality
\begin{align}\label{def-n_E}
n_E:=\max_{\substack {\lambda, \lambda^{\prime}\in E \\  \lambda\neq \lambda^{\prime}} } v_p(\lambda- \lambda^{\prime}).
\end{align}
The following proposition characterizes the structure of  $\ZE$, the set of zeros of the
Fourier transform of the discrete  measure $\delta_E$, i.e. it is bounded and is a union of spheres centered at $0$.
\begin{prop}[\cite{FFLS}, Proposition 2.9]\label{zeroofE} 
Let $E$ be a  uniformly discrete set in $\Q_p$.\\
\indent {\rm (1)} 
If $\xi\in \ZE$,  then $S(0,|\xi|_p)\subset \ZE$.\\
\indent {\rm (2)} The set $\ZE$ is bounded. Moreover, we have
\begin{align}\label{nE}
\ZE\subset   B(0,p^{n_E+1}).
\end{align}
\end{prop}



\subsection{Convolution and  multiplication of distributions}
Hereafter, the following notations are frequently used:
 $$
\Delta_k:=1_{B(0,p^k)}, \quad
\theta_k:=\widehat{\Delta}_k=p^k \cdot 1_{B(0,p^{-k})}.
$$
Let $f,g\in \mathcal{D}'$
 be two distributions. We define the {\em convolution} of $f$ and $g$ by
$$
\langle f*g,\phi \rangle =\lim\limits_{k\to \infty}	 \langle f(x), \langle g(\cdot),\Delta_k(x)\phi(x+\cdot) \rangle \rangle,
$$
if the limit exists for all $\phi \in \mathcal{D}$. 

{
\begin{prop} [\cite{Aks}, Proposition  4.7.3] If $f\in \mathcal{D}^{\prime}$, then  $f*\theta_k\in \mathcal{E}$
with  the parameter of constancy at least  $-k$.
\end{prop}
}

We define the
{\em  multiplication} of $f$ and $g$ by
$$
\langle f\cdot g,\phi \rangle =\lim\limits_{k\to \infty}	 \langle g, ( f*\theta_k)\phi  \rangle,
$$
if the limit exists for all $\phi \in \mathcal{D}$.        The above definition of convolution is
compatible with the usual convolution of two integrable functions and the definition of multiplication is compatible with  the usual  multiplication of two locally integrable functions.

The following proposition shows that both the convolution and the multiplication are commutative whenever they are well defined and the convolution of two distributions is well defined if and only if the multiplication 
of their Fourier transforms is well defined. 

\begin{prop} [\cite{Vvz}, Sections 7.1 and 7.5] \label{Conv-Mul} Let $f, g\in \mathcal{D}'$ be two distributions. Then\\
\indent {\rm (1)} \  If $f*g$ is well defined, so is $g*f$ and  $f*g=g*f$.\\
\indent {\rm (2)} \ If $ f\cdot g$ is well defined, so is  $g\cdot f$
 and  $ f\cdot g= g\cdot f$.\\
 \indent {\rm (3)} \ $f*g$ is well defined  if and only   $\widehat{f}\cdot \widehat{g}$ is well defined. In this case, we have
$
\widehat{f*g}=\widehat{f}\cdot\widehat{g}$ and
$ \widehat{f\cdot g}=\widehat{f}*\widehat{g}.
$
\end{prop}

The following proposition justifies an intuition which is crucial in the proof of Fuglede's conjecture on $\Qp$ (see \cite{FFLS} for more details). 

 \begin{prop}[\cite{FFLS}, Proposition 2.12]  \label{zeroproduct}
Let  $f, g\in \mathcal{D}'$ be two distributions. If $\operatorname{supp}(f) \cap \operatorname{supp}(g) =\emptyset $, then $f\cdot g$ is well defined and 
$f\cdot g =0$.
\end{prop}


The multiplication of some special distributions   has a simple form. That is the case for  the multiplication of a uniformly locally constant function and a distribution.

\begin{prop}[\cite{Vvz} Section 7.5, Example 2]\label{prod}
	Let $f\in \mathcal{E}$ and let $G\in \mathcal{D}'$. Then for any $\phi\in \mathcal{D}$, we have
	$
	\langle f\cdot G, \phi \rangle = \langle G, f\phi \rangle.
	$
\end{prop}

 For a distribution $f\in \mathcal{D}'$, we define its
 {\em regularization} by the sequence of test functions  (\cite[ Proposition  4.7.4]{Aks})
$$
\Delta_k \cdot(f*\theta_k) \in \mathcal{D}.
$$
The regularization of a distribution  converges  to the distribution itself
with respect to the weak topology.

\begin{prop}[\cite{Aks} Lemma 14.3.1]\label{limit}
	Let $f$ be a distribution in $\mathcal{D}'$.
	Then $\Delta_k \cdot (f*\theta_k)\to f$ in $\mathcal{D}'$ as $k\to \infty$. Moreover, for any test function $\phi\in \mathcal{D}$ we have
	$$
	\langle \Delta_k \cdot(f*\theta_k), \phi \rangle=\langle f, \phi \rangle, \quad \forall k\ge \max\{-\ell,\ell^{\prime}\},
	$$
	where $\ell$ and $\ell^{\prime}$ are the parameter of constancy and the  parameter of compactness of the function $\phi$ defined in Subsection \ref{subsec2.4}. 
\end{prop}

This approximation of distribution by test functions allows us to construct a   space which is bigger than the space of distributions. This larger space is the Colombeau algebra, which will be presented below. Recall that in the space of Bruhat-Schwartz distributions, the convolution and the multiplication are not well defined for all couples of distributions.  But in the Colombeau algebra,  the convolution and the multiplication are well defined  and these two operations are associative. 

\subsection{Colombeau algebra of generalized functions}
Consider the space $\mathcal{P}: = \mathcal{D}^\mathbb{N}$ of all sequences $\{f_k\}_{k\in \mathbb{N}}$ of test functions.  We introduce an algebra
structure on $\mathcal{P}$, by defining the operations component-wisely
$$
\{f_{k}\}+\{ g_{k}\}=\{f_{k}+g_{k}\},
$$
$$
\{f_{k}\}\cdot\{g_{k}\}=\{f_{k}\cdot g_{k}\},
$$
where $\{f_{k}\}, \{g_{k}\}\in \mathcal{P}$. 

 Let $\mathcal{N}$ be the sub-algebra of elements $\{f_k\}_{k\in \mathbb{N}}\in \mathcal{P}$  such that for any compact set $K\subset\Q_p$ there exists $N\in \mathbb{N}$ such that   $f_{k}(x)=0$ for all $k\geq N, x\in K$. Clearly, $\mathcal{N}$ is an ideal in the algebra $\mathcal{P}$.   
The  {\em  Colombeau-type algebra} is defined by the quotient
$$
\mathcal{G}=\mathcal{P}/\mathcal{N}.
$$
 The equivalence class of sequences which defines an element in  $ \mathcal{G}$ will be denoted by $\mathbf{f}=[f_k]$, called a \textit{generalized function}.
 
 For any $\mathbf{f}=[f_k], \mathbf{g}=[g_k]\in \mathcal{G}$, the addition   and multiplication  are defined as  
 $$
 \quad \mathbf{f}+\mathbf{g}=[f_k+g_k],  \quad \mathbf{f}\cdot\mathbf{g}=[f_k\cdot g_k].
$$

Obviously,  $(\mathcal{G}, +, \cdot)$     is an associative and commutative algebra.

\begin{thm}[\cite{Aks} Theorem 14.3.3]\label{thm:embedding}
	The map $f\mapsto \mathbf{f}=[\Delta_k (f*\theta_k)]$ from  $\mathcal{D}'$  to $\mathcal{G}$ is  a linear embedding.
\end{thm}

Each distribution $f\in \mathcal{D}'$ is embedded into 
$\mathcal{G}$ by the mapping which associates $f$ to the generalized function determined by the regularization of $f$. 
Thus we obtain that the multiplication  defined on the $\mathcal{D}'$ is associative in the following sense.
\begin{prop}[\cite{FFLS}, Proposition 2.16]\label{associative}
 Let $f,g,h\in \mathcal{D}'$. If $(f\cdot g)\cdot h$ and $f\cdot (g\cdot h)$ are well defined as multiplications of distributions, we have  $$(f\cdot g)\cdot h =f\cdot (g\cdot h).$$

\end{prop}

\section{Dimensions of measures and sets}\label{sec:dimension}
In this section, we recall some basic notions and properties of dimensions of measures and sets.

\subsection{Dimensions of sets}

In general, the Hausdorff dimension is well defined in any metric space. For convenience, we only state the definition of Hausdorff dimension in the space of $p$-adic numbers as follows. The \textit{Hausdorff dimension} of a Borel subset $\Omega\subset \Qp^d$ is defined by
$$
\dim_{\text{H}}\Omega=\inf\{\alpha: \sum_{k=1}^{\infty} \text{diam}(A_k)^\alpha=0,~\Omega\subset \cup_k A_k  \},
$$
where $\text{diam}(A)$ denotes the diameter of the set $A$.

\subsection{Dimensions of measures}
Now we recall several notions of dimensions of measures. Let $\mu \in \mathcal{P}(\Qp^d)$ where $\mathcal{P}(\Qp^d)$ denotes the space of probability measures in $\Qp^d$. Let $\mathcal{A}$ be a partition of $\Qp^d$. The \textit{Shannon entropy of $\mu$ with respect to $\mathcal{A}$} is defined by
$$
H(\mu,\mathcal{A})=\sum_{A\in \mathcal{A}}-\mu(A)\log \mu(A).
$$
By convention the logarithm is taken in base $p$ and $0\log 0=0$. If the partition $\mathcal{A}$ is infinite, then the entropy $H(\mu, \mathcal{A})$ may be infinite. We consider a special family of partitions, where the $n$-th $p$-adic partition of $\Qp^d$ is defined by
$$
\mathcal{D}_n:=\left\{ z+p^n\Zp: z\in \Z/p^{n}\Z  \right\},
$$
which consists of compact open balls of radius $p^n$.  The \textit{entropy dimension} of $\mu$ is defined by
$$
\dim_e \mu=\lim\limits_{n\to \infty} \frac{1}{n} H(\mu, \mathcal{D}_n),
$$
if the limit exists (otherwise we take limsup or liminf as appropriate, denoted by $\overline{\dim}_e \mu$ and $\underline{\dim}_e \mu$).

The \textit{lower local dimension} of a measure $\mu$ at $x$ is defined by the formula
$$
\underline{d}(\mu, x)=\liminf\limits_{r\to 0} \frac{\log \mu(B(x,r))}{\log r}.
$$
Similarly, the \textit{upper local dimension} is defined by
$$
\overline{d}(\mu, x)=\limsup\limits_{r\to 0} \frac{\log \mu(B(x,r))}{\log r}.
$$
The \textit{lower Hausdorff dimension} of $\mu$ is defined by
$$
\dimLow_H \mu=\inf\{\dim_H A: \mu(A)>0 \},
$$
and the \textit{upper Hausdorff dimension} of $\mu$ is 
$$
\overline{\dim}_H \mu=\inf\{\dim_H A: \mu(A)=1 \}.
$$
It is well known that
\begin{equation}\label{eq:dimension fact}
\dimLow_H \mu=\text{essinf}_{x\thicksim\mu} {\underline{d}(\mu,x)}~\text{and}~\overline{\dim}_H \mu=\text{esssup}_{x\thicksim\mu} {\underline{d}(\mu,x)}.
\end{equation}
We denote the common value by $\dim_H \mu$ if $\dimLow_H \mu=\overline{\dim}_H \mu$.

The following proposition is well known in the setting $\R^d$. The proof in the setting $\Qp^d$ is similar with $\R^d$. Thus we omit the proof and leave the readers to work out the detail.
\begin{prop}\label{prop:local dimension}
	Let $\mu\in \mathcal{P}(\Qp^d)$ and $\Omega\subset \Qp^d$ a set with $\mu(\Omega)>0$. Then we have the following properties.
	\begin{itemize}
		\item [(1) ] (Mass distribution principle) If $\underline{d}(\mu,x)\ge \alpha$ for all $x\in \Omega$, then $\dim_{\text{H}}\Omega\ge \alpha$. 
		\item [(2) ] (Billingsley's lemma) If $\underline{d}(\mu,x)\le \alpha$ for all $x\in \Omega$, then $\dim_{\text{H}}\Omega\le \alpha$.
	\end{itemize}
\end{prop}

A direct consequence of Proposition \ref{prop:local dimension} is that if $\underline{d}(\mu, x)=\alpha$ for all $x\in \Omega$ then $\dim_{\text{H}}\Omega=\alpha$. In fact, $\underline{d}(\mu, x)=\alpha$ means that the decay of $\mu$-mass of balls centered at $x$ scales no slower than $r^\alpha$, in other words, for every $\epsilon>0$, we have $\mu(B(x,r))\le r^{\alpha-\epsilon}$ for all small enough $r$ and that such $\alpha$ is the largest number satisfying this property.

\subsection{Beurling dimension of countable sets}\label{Sec:Beurling dimension of countable sets}
Let $\Lambda$ be a countable set in $\Qp^d$. For $r>0$, the \textit{upper Beurling density corresponding to $r$} (or \textit{$r$-Beurling density}) is defined by the formula
$$
\mathfrak{D}_r^+:=\limsup_{h\to \infty} \sup_{x\in \Qp^d} \frac{\sharp(\lambda \cap B(x,h))}{h^r}.
$$
Similarly, the \textit{lower Beurling density corresponding to $r$} is defined by
$$
\mathfrak{D}_r^-:=\liminf_{h\to \infty} \inf_{x\in \Qp^d} \frac{\sharp(\lambda \cap B(x,h))}{h^r}.
$$

The \textit{(upper) Beurling dimension} is defined by
$$
\dim_B\Lambda=\sup\{r>0:\mathfrak{D}_r^+(\Lambda)>0 \},
$$
or alternatively,
$$
\dim_B\Lambda=\inf\{r>0:\mathfrak{D}_r^+(\Lambda)<+\infty \}.
$$

\section{Finite $p$-homogeneous set in $\Qp$}\label{sec:tree structure}
In \cite{FFS}, Fan et al. defined $p$-homogeneous sets and $p$-homogeneous trees in finite groups $\Z/p^{\gamma}\Z$. Inspired by these concepts, we introduce the notion of finite $p$-homogeneous sets in $\Qp$. We first recall the definitions of $p$-homogeneous sets and $p$-homogeneous trees in $\Z/p^{\gamma}\Z$.

Let $\gamma $ be a positive integer. 
We identify  $ \mathbb{Z}/p^\gamma\mathbb{Z}\simeq\{0, 1, \cdots, p^\gamma-1\}$  with $\{0, 1, 2,  \cdots, p-1\}^\gamma$ which is considered as a finite tree, denoted by $\mathcal{T}^{(\gamma)}$. In fact, the vertices of $\mathcal{T}^{(\gamma)}$ are the sets $\mathbb{Z}/p^n\mathbb{Z}, 0\leq n\leq \gamma$, translations of them and emptyset $\emptyset$ which is regarded as the root of the tree. In other words, each  vertex, except the root of the tree,  is identified with a sequence $t_0t_1\cdots t_{n-1}$ with $1\leq n\leq\gamma$ and $t_i\in \{0,1, \cdots, p-1\}$. Here, we identify $t_0t_1\cdots t_{n-1}$ with $\sum_{i=0}^{n -1} t_i p^i+\Z/p^{n}\Z$, the translation of the set $\Z/p^n\Z$ by $\sum_{i=0}^{n -1} t_i p^i$.
The set of edges consists of pairs $(x,y)\in \mathbb{Z}/p^n\mathbb{Z}\times \mathbb{Z}/p^{n+1}\mathbb{Z}$, such that $x\equiv y ~(\!\!\!\!\mod p^n)$, where $0\leq n\leq \gamma-1$. Moreover, each
point $c$ of $\mathbb{Z}/p^\gamma\mathbb{Z}$ is identified with $\sum_{i=0}^{\gamma -1} t_i p^i \in \{0, 1, \cdots, p^\gamma-1\}$, which is called {\em a boundary point} of the tree. See Figure \ref{Fig:tree}.
Thus each subset $C\subset \mathbb{Z}/p^\gamma\mathbb{Z}$ will determine a subtree of $\mathcal{T}^{(\gamma)}$, denoted by $\mathcal{T}_{C}$, which consists  of 
the paths from the root  to the boundary points in $C$.

\begin{figure}[h!]
	\centering
	\includegraphics[width=0.8\textwidth]{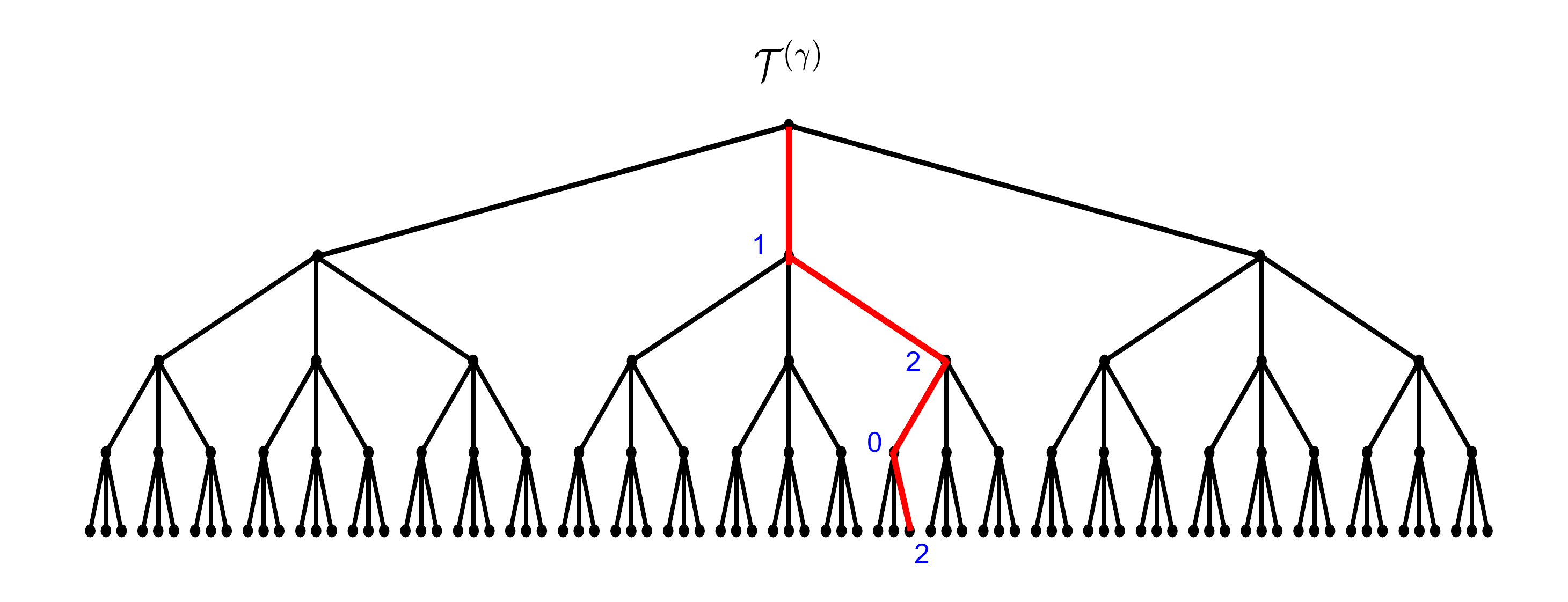}\\
	\caption{$p=3, \gamma=4, c=1\cdot 3^0+2\cdot 3^1+0\cdot 3^2+2\cdot 3^3$.} 
	\label{Fig:tree}
\end{figure}

Now we are going to construct a special class of subtrees of $\mathcal{T}^{(\gamma)}$.
Let $I$ and $J$ form a
partition of $\{0, 1, 2, \cdots, \gamma-1\}$.
It is allowed that either $I$ or $J$ is empty.
We say a subtree of $\mathcal{T}^{(\gamma)}$ is of  {\em $\mathcal{T}_{I, J}$-form} if its boundary points
$t_0t_1\cdots t_{\gamma-1}$ of  $\mathcal{T}_{I, J}$ are those of $\mathcal{T}^{(\gamma)}$ satisfying the following conditions:\\
\indent (i) \ \  if $i \in I$, $t_i$ can take any value of $\{0, 1, \cdots, p-1\}$;\\
\indent (ii) \ if $i\in J$, for any $t_0t_1\cdots t_{i-1}$, we fix one value of $\{0, 1, \cdots, p-1\}$ which is the only value taken by $t_i$. In other words, $t_i$ takes only one value from $\{0, 1, \cdots, p-1\}$ which depends on $t_0t_1\cdots t_{i-1}$.

Remark that such a subtree depends not only on $I$ and $J$ but also on the values taken
by $t_i$'s with $i\in J$. 
A $\mathcal{T}_{I, J}$-form tree  is  called a  finite $p$-{\em homogeneous tree}. A set $C\subset \Z/p^{\gamma}\Z$ is said to be {\em $p$-homogeneous} if the corresponding tree $\mathcal{T}_{C}$ is  $p$-homogeneous. See examples in Figure \ref{Fig:homogeneous tree}.

\begin{figure}[h!]
	\centering
	\includegraphics[width=0.8\textwidth]{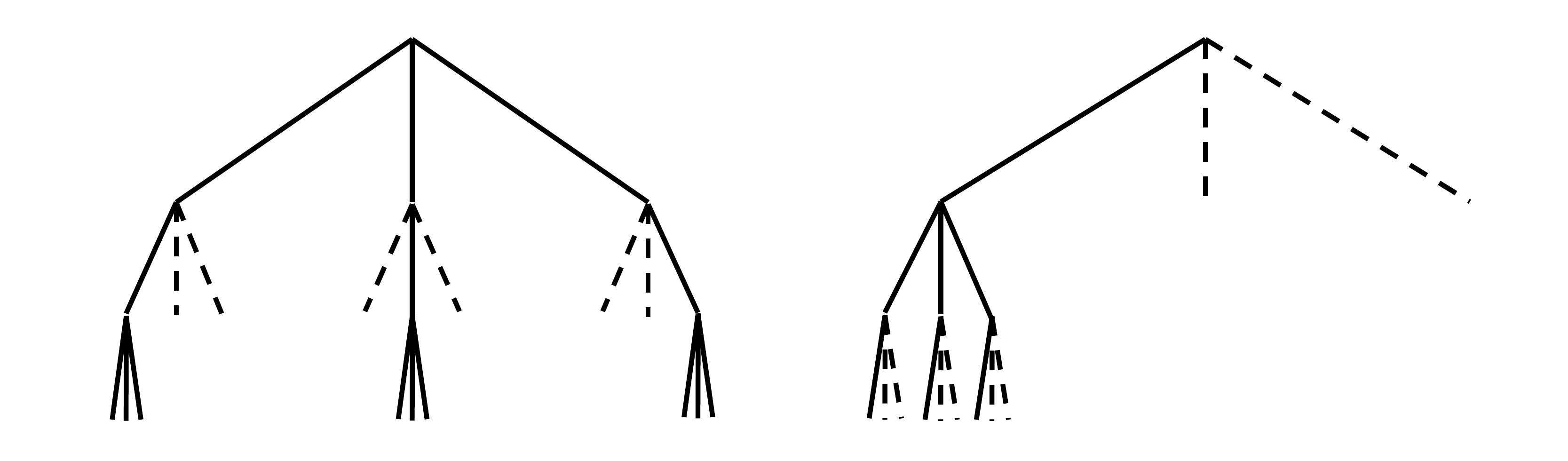}\\
	\caption{$p$-homogeneous trees $\mathcal{T}_C$ and $\mathcal{T}_D$ in $\Z/3^3\Z$, where $C=\{0,4,8,9,13,17,18,22,26 \}$ and $D=\{0,3,6\}$ are contained in $\Z/3^3\Z$. Moreover, $\mathcal{T}_C$ is of $\mathcal{T}_{I,J}$ form and $\mathcal{T}_D$ is of $\mathcal{T}_{J,I}$ where $I=\{0,2\}$ and $J=\{1\}$.} 
	\label{Fig:homogeneous tree}
\end{figure}

For $C\subset \Z_p$, denote   by
$$C_{\!\!\!\!\!\! \mod{p^\gamma}} :=\{x\in \{0,1,\dots,p^\gamma-1 \}: \exists~ y \in C, \text{ such that }  x\equiv y ~(\!\!\!\!\!\! \mod{p^\gamma}) \}$$  the multi-subset of  $\Z/p^\gamma\Z$ determined by $C$ modulo $p^\gamma$. 
Now we introduce the notion of finite $p$-homogeneous sets in $\Qp$: a finite set $C$ is said to be {\em $p$-homogeneous} if there exist $n,\gamma\in \Z$ such that
$
(p^nC)_{\!\!\!\! \mod{p^\gamma}}
$
is $p$-homogeneous in $\Z/p^{\gamma}\Z$. Moreover, we say that a finite $p$-homogeneous set $C$ has the \textit{$(n, \gamma, I, J)$-tree structure} if the corresponding tree $\mathcal{T}_{(p^nC)_{\!\!\!\! \mod{p^\gamma}}}$ is a $\mathcal{T}_{I,J}$-form tree. It is easy to see that if $(p^nC)_{\!\!\!\! \mod{p^\gamma}}$ is $p$-homogeneous in $\Z/p^{\gamma}\Z$, then $(p^{n+k}C)_{\!\!\!\! \mod{p^{\gamma+k}}}$ is $p$-homogeneous in $\Z/p^{\gamma+k}\Z$ for any $k\ge 0$. Thus if the set $C$ has $(n, \gamma, I, J)$-tree structure, then it also has $(n+k, \gamma+k, I+k, (J+k)\cup \{0, 1, \dots, k-1 \})$-tree structure for any $k\ge 1$, where $I+k=\{i+k:i\in I  \}$ is the translation of $I$ by $k$. See Figure \ref{Fig:p-homo set} for an example. We can understand the parameters $n, \gamma, I$ and $J$ in the following way:
\begin{itemize}
	\item ``$n$": the position in $\Qp$ where we look at;
	\item ``$\gamma$": the mass of information that we deal with;
	\item ``$I$" and ``$J$": the form of the set, i.e. how the set looks like. Moreover, the larger $\sharp I$ is, the large the set is.
\end{itemize}

\begin{figure}[h!]
	\centering
	\includegraphics[width=0.8\textwidth]{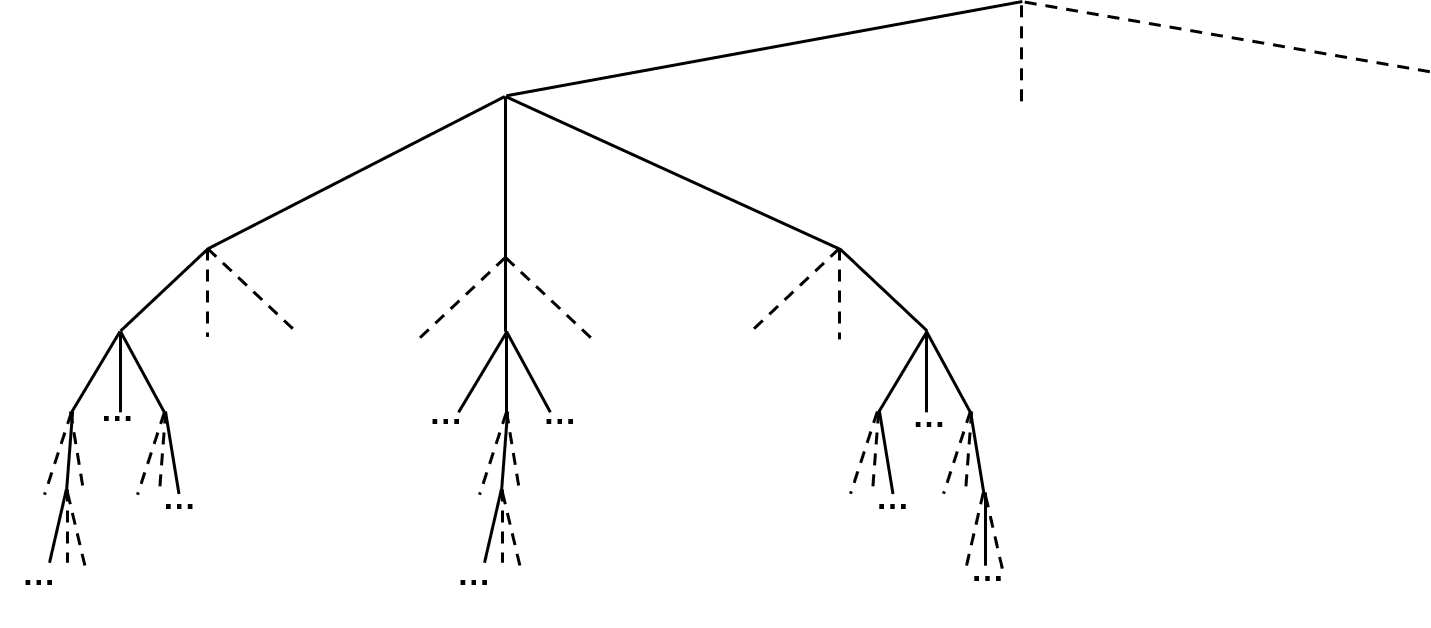}\\
	\caption{ A $p$-homogeneous set with $(n, \gamma, I, J)$-tree structure, where $I=\{k,2+k\}$ and $J=\{0, 1, 2, \dots, k-2, k-1, k+1 \}$ for some $k\ge 0$ (dependent on the choice of $n,\gamma$).} 
	\label{Fig:p-homo set}
\end{figure}

\begin{lem}\label{lem:p homo spectral}
	Let $C$ be the $p$-homogeneous set having $(n, \gamma, I, J)$-tree structure. Let $D$ the $p$-homogeneous set having $(\gamma-n, \gamma, \gamma-I, \gamma-J)$-tree structure. Then the pair $(C,D)$ is a spectral pair.
\end{lem}
\begin{proof}
	It is not hard to check that the matrix
	$$
	H:=\left( \chi(cd) \right)_{c\in C, d\in D}
	$$
	is a complex Hadamard matrix, that is to say, $HH^\dagger=(\sharp C) I$, where $H^\dagger$ denotes the Hermitian transpose of $H$ and $I$ is the identity matrix.
\end{proof}
Finally, considering the tree model of $p$-adic field, we remark that the local geometrical structure of the $p$-homogeneous sets $C$ and $D$ in Lemma \ref{lem:p homo spectral} is like in Figure \ref{Fig:homogeneous tree}.

\section{Zeros of distributions}\label{sec:zeros}
In this section, we investigate the relation between the zeros of two distributions and the ones of their product. We say that a distribution $f\in \mathcal{D}'$ is \textit{non-negative} if for any non-negative test function $\phi\in \mathcal{D}$, we have $\langle f, \phi\rangle\ge 0$.

We show in the following lemma that a non-negative distribution has a local inverse generalized function in its support.
\begin{lem}\label{lem:local inverse}
	Let $f\in \mathcal{D}'$ be a non-negative distribution. Suppose that the support of $f$ contains a ball $B(x,p^{-n})$ for some $x\in \Q_p$ and some integer $n\ge 0$, namely,
	$$
	B(x,p^{-n}) \subset \text{supp} (f).
	$$
	Then there exists a generalized function $\mathbf{g}\in \mathcal{G}$ such that $\mathbf{g} \cdot f =1_{B(x,p^{-n})}$.
\end{lem}
\begin{proof}
	We claim that for any $y\in B(x,p^{-n})$ and $m\ge n$, we have $\langle f, 1_{B(y,p^{-m})} \rangle>0$. We prove this claim by contradiction. Assume that there is some ball $B(y,p^{-m})\subset B(x, p^{-n})$
	such that $\langle f, 1_{B(y,p^{-m})} \rangle=0$. It follows that
	\begin{equation}\label{eq:local inverse 1}
	\langle \Delta_k \cdot(f*\theta_k), 1_{B(y,p^{-m})} \rangle=\langle f, 1_{B(y,p^{-m})} \rangle=0,~\forall~k\ge m.
	\end{equation}
	The first equality above is due to Proposition \ref{limit}. On the other hand, for $k\ge m$, we have
	\begin{equation}\label{eq:local inverse 2}
	\langle \Delta_k \cdot(f*\theta_k), 1_{B(y,p^{-m})} \rangle=\sum_{i\in I_k} \langle \Delta_k \cdot(f*\theta_k), 1_{B(y_i,p^{-k})} \rangle,
	\end{equation}
	where $I_k$ is a finite set such that $\{B(y_i, p^{-k}): i\in I_k \}$ is a partition of $B(y,p^{-m})$. Since $f$ is non-negative, by \eqref{eq:local inverse 1}, \eqref{eq:local inverse 2} and Proposition \ref{limit}, we have
	\begin{equation}
	\langle \Delta_k \cdot(f*\theta_k), 1_{B(y_i,p^{-k})} \rangle=0,~\forall~k\ge m~\forall~i\in I_k.
	\end{equation}
	Observe that for any $z\in B(y,p^{-m})$, the ball $B(z,p^{-k})$ coincides one of $B(y_i, p^{-k})$ for every $k\ge m$.
	Thus for any $z\in B(y,p^{-m})$, we have
	$$
	\langle f, 1_{B(z,p^{-k})} \rangle= \langle \Delta_k \cdot(f*\theta_k), 1_{B(z,p^{-k})}\rangle=0, \forall k\ge m.
	$$
	This implies that $y\in \mathcal{Z}_{f}$ which is contradict to the hypothesis. Thus we prove the claim.
	
	Since $\Delta_k \cdot(f*\theta_k)\in \mathcal{D}$, by the above claim, we have $(\Delta_k \cdot(f*\theta_k))(\xi)\not=0$ for all $\xi\in B(x,p^{-n})$ and all $k\ge n$. Now define 
	\begin{equation*}
	g_k(\xi)=
	\begin{cases}
	((\Delta_k \cdot(f*\theta_k))(\xi))^{-1},&~\text{if}~\xi\in B(x,p^{-n}),\\	
	0,&~\text{otherwise},
	\end{cases}
	\end{equation*}
	for $k\ge n$ and $g_k=0$ for $1\le k<n$. Since $\Delta_k \cdot(f*\theta_k)\in \mathcal{D}$, we have $g_k\in \mathcal{D}$ for all $k\ge 1$. Let $\mathbf{g}=[g_k] \in \mathcal{G}$. A simple computation shows that
	$$
	g_k \cdot (\Delta_k \cdot(f*\theta_k)) =
	\begin{cases}
	1_{B(x,p^{-n})},&~\text{if}~k\ge n,\\
	0,&~\text{if}~1\le k\le n-1.
	\end{cases}
	$$
	It follows that
	\begin{equation*}
	\mathbf{g}\cdot  f=[ g_k \cdot(\Delta_k \cdot(f*\theta_k)) ]=[\Delta_k \cdot(1_{B(x,p^{-n})}*\theta_k)].
	\end{equation*}
	By Theorem \ref{thm:embedding}, we conclude that $\mathbf{g}\cdot f=1_{B(x,p^{-n})}$.
\end{proof}

The following proposition tells us that the union of the set of zeros of two  distributions contains the set of zeros of their product whenever at least one of them is non-negative. We remark that if neither of them is non-negative, then such relation might not hold in general.
\begin{prop}\label{prop:union of zeros}
	Let $f,g\in \mathcal{D}'$. Suppose $f$ is non-negative. If the product $f\cdot g$ is well defined and equal to $h$, then we have
	$$
	\mathcal{Z}_{h} \subset \mathcal{Z}_{f} \cup \mathcal{Z}_{g}.
	$$
\end{prop}
\begin{proof}
	It is sufficient to show 
	\begin{equation}\label{eq:unions of zeros 0}
	\mathcal{Z}_{h}\setminus \mathcal{Z}_f \subset \mathcal{Z}_g.
	\end{equation}
	This means that we only need to show that for any $\phi\in \mathcal{D}$ with $\text{supp}(\phi)\subset \mathcal{Z}_{h}\setminus \mathcal{Z}_f$, we have $\langle g, \phi\rangle=0$. Now fix $\phi=1_{B(x,p^{-n})}$ with $\text{supp}(\phi)\subset \mathcal{Z}_{h}\setminus \mathcal{Z}_f$. It follows that 
	$$
	\text{supp}(\phi)\subset \text{supp}(f).
	$$
	By Lemma \ref{lem:local inverse}, there exists $\mathbf{w}\in \mathcal{G}$ such that $\mathbf{w} \cdot f=\phi$. By Proposition \ref{prod} and the fact that $\phi^2=\phi$, we have
	\begin{equation}\label{eq:unions of zeros 1}
	\langle g, \phi \rangle=\langle g, \phi^2 \rangle=\langle \phi g, \phi \rangle=\langle \mathbf{w} \cdot f \cdot  g, \phi \rangle=\langle \mathbf{w} \cdot h, \phi \rangle.
	\end{equation}
	By the construction of $\mathbf{w}=[w_k]$ in Lemma \ref{lem:local inverse}, we might assume that the parameter of constancy of $w_k$ is $-k$ and 
	$$
	\text{supp}(w_k)\subset \text{supp}(\phi),~\forall~k\ge 1.
	$$
	For any $\varphi\in \mathcal{D}$ and any $k\ge 1$, since $\text{supp}(w_k\cdot\varphi) \subset \text{supp}(\phi) \subset \mathcal{Z}_h$, by Proposition \ref{prod} and Proposition \ref{limit}, we have
	$$
	\langle w_k\cdot(\Delta_k \cdot(h*\theta_k)), \varphi \rangle=\langle \Delta_k \cdot(h*\theta_k), w_k\cdot\varphi \rangle=\langle h, w_k\cdot\varphi \rangle=0.
	$$
	It follows from Theorem \ref{thm:embedding} that $\mathbf{w}\cdot h=0$. Combining this with \eqref{eq:unions of zeros 1}, we conclude that $\langle g, \phi \rangle=0$ which justifies \eqref{eq:unions of zeros 0}.
\end{proof}
Proposition \ref{prop:union of zeros} provides us a useful tool to investigate the functional equation of distributions with the form $f\cdot g=h$. We will use it later in Section \ref{sec: spectral measures in Qp}.

\section{$p^n$-cycle in $\Qp$} \label{Zmodule}

In this section, we first recall the notion of $p$-cycles in cyclic groups and then introduce the notion of $p^n$-cycles in the field of $p$-adic numbers. It is an important tool to deal with the vanishing sum of continuous group characters.

Let $m\ge 2$ be an integer and let  $\omega_m = e^{2\pi i/m}$, which is a primitive $m$-th root of unity. Denote  by $\mathcal{M}_m$  the set of integral points
$(a_0, a_1, \cdots, a_{m-1}) \in \mathbb{Z}^m$ such that
$$
\sum_{j=0} ^{m-1} a_j \omega_m^j =0.
$$
The set $\mathcal{M}_m$ is clearly a $\mathbb{Z}$-module. Throughout this section, we are concerned with the case where  $m=p^n$ is a power of a prime number. 
The structure of   $\mathcal{M}_{p^n}$ is shown in the following lemma.
\begin{lem} [\cite{s}, Theorem 1]\label{SchLemma}
	If $(a_0,a_1,\cdots, a_{p^n-1})\in  \mathcal{M}_{p^n}$,
	then for any integer $0\le i\le p^{n-1}-1$ we have $a_i=a_{i+jp^{n-1}}$ for all $j=0,1,\dots, p-1$.
\end{lem}
Lemma \ref{SchLemma} has the following special form. 
\begin{lem}
	\label{permu}
	Let $(b_0,b_1,\cdots, b_{p-1})\in \mathbb{Z}^{p}$.
	If $\sum_{j=0}^{p-1} e^{2\pi i b_j/p^n}=0$, then subject to a permutation of $(b_0,\cdots, b_{p-1})$, there exist $0\leq r \leq p^{n-1}-1$ such  that
	$$b_j  \equiv r+ jp^{n-1} (\!\!\!\!\mod p^n) $$ for all  $ j =0,1, \cdots,p-1$.
\end{lem}
The set $\{b_0, b_1, \dots, b_{p-1} \}$ is sometime called a \textit{$p$-cycle in $\Z$} (or in $\Z/p^n\Z$).
Now we introduce the notion of $p^n$-cycles in $\Qp$.
We say that a finite set $C\subset \Qp$ is \textit{a $p^n$-cycle in $\Qp$} if $C=\{c_0, c_1, \cdots, c_{p-1}  \}$ with the form
$$
c_j=\varsigma+jp^{n}+\zeta_j, 
$$
for $0\le j\le p-1$ where $\varsigma\in \Qp$ and $\zeta_j\in B(0, p^{-n-1})$. This means that $\left( C_{\mod p^n} \right)$ is a $p$-cycle in $\Z$ (or alternately in $\Z/p^{n+1}\Z$). 
The $p^n$-cycles play an important role in the study of vanishing sum of characters, which is shown as follows.
\begin{lem}\label{lem:p^n-cycles}
	Let $C\subset\Q_p$ be a finite set. There exists $\xi\in \Qp$ such that 
	$\sum_{c\in C}\chi(\xi c)=0,$ if and only if $C$ is a union of $\frac{p}{|\xi|_p}$-cycles. 
\end{lem}
\begin{proof}
By Lemma \ref{SchLemma}, we obtain that $C$ is a union of the sets satisfying the following conditions: each of them, say $C_k$, satisfies that $\sharp C_k=p$ and that
\begin{equation*}
\sum_{c\in C_k} \chi(\xi c)=0.
\end{equation*}
By Lemma \ref{permu} and the definition of $p^n$-cycles, we have that each $C_k$ is a $p^n$-cycle. This completes the proof.
\end{proof}

For a finite set $C\in \Qp$, we denote by $\sharp C_n^{\xi}=\sharp (C\cap B(\xi, p^n))$.

\section{Density of uniformly discrete set}\label{sec:density}
We say that a uniformly discrete set  $E$ has a {\em bounded density} if the following limit exists for some $x_0\in \Q_p$	
$$
D(E):=\lim\limits_{k\to \infty}\frac{\sharp(B(x_0,p^k)\cap E)}{p^k},
$$
which is called the {\em density} of $E$. Actually, if the limit exists for some $x_0 \in \Q_p$, then
it exists for all $x\in \Q_p$ and the limit is independent of $x$. In fact, for any $x_0, x_1\in \Qp$, when $k$ is  large enough such that  $|x_0-x_1|_p<p^k$, we have $B(x_0, p^k)=B(x_1, p^k)$ eventually.

\begin{prop}\label{prop:density zero}
	Let $f\in \mathcal{D}'$ be non-negative. Let $E$ be a uniformly discrete set in $\Qp$. Suppose $f*\delta_E=1$ and $f\notin L^1(\Qp)$. Then the density of $E$ exists and $D(E)=0$.
\end{prop}
\begin{proof}
Since $f$ is non-negative and $f\notin L^1(\Qp)$, there exists an integer $N$ such that $\langle f, 1_{B(0,p^n)} \rangle>0$ for all $n\ge N$.
Integrating the equality $f*\delta_E=1$ over the ball $B(0, p^n)$, we have
\begin{eqnarray}\label{ee}
p^n=\sum_{\lambda \in E} \langle f(\cdot-\lambda), 1_{B(0,p^n)}(\cdot)\rangle.  
\end{eqnarray}	
Observe that 
\begin{equation}\label{eq:density zero 1}
\langle f(\cdot-\lambda), 1_{B(0,p^n)}(\cdot)\rangle=\langle f, 1_{B(-\lambda,p^n)}\rangle=\langle f, 1_{B(y,p^n)}\rangle
\end{equation}
for any $y\in B(-\lambda, p^n)$.
Let $P_n$ be the  set in $\Qp$ such that $\{B(y, p^n): y\in P_n \}$ is a partition of $\Qp$. 
By \eqref{ee} and \eqref{eq:density zero 1}, we have
\begin{equation}\label{eq:density zero 2}
1=\sum_{y \in P_n} \frac{\sharp(E\cap B(y,p^n) )}{p^n}\cdot \langle f, 1_{B(y,p^n)}\rangle. 
\end{equation}
Without loss of generality, we might assume $0\in P_n$ for all $n$. By \eqref{eq:density zero 2} and the fact that $f$ is non-negative, we have 
\begin{equation*}
1\ge \frac{\sharp(E\cap B(0,p^n) )}{p^n}\cdot \langle f, 1_{B(0,p^n)}\rangle,~\forall~n\ge N. 
\end{equation*}
This implies that
\begin{equation}
\frac{\sharp(E\cap B(0,p^n) )}{p^n}\le \frac{1}{\langle f, 1_{B(0,p^n)}\rangle},~\forall~n\ge N.
\end{equation}
Since $f\notin L^1(\Qp)$ and $f$ is non-negative, we finally get 
$$
\limsup\limits_{n\to \infty} \frac{\sharp(E\cap B(0,p^n) )}{p^n}\le \lim\limits_{n\to \infty}\frac{1}{\langle f, 1_{B(0,p^n)}\rangle} =0.
$$	
\end{proof}
We remark that the case where $f\in L^1(\Qp)$ (not necessarily non-negative) has been considered in \cite{FFLS}.

\section{Spectral measures in $\Qp$}\label{sec: spectral measures in Qp}
Let $\mu$ be a spectral measure with spectrum $\Lambda$, which is characterized by the following functional equation
\begin{equation}\label{F-equation}
|\widehat{\mu}|^2 * \delta_{\Lambda} = 1.
\end{equation}
We remark that the convolution in (\ref{F-equation})  is understood as a convolution of Bruhat-Schwartz distributions and even as a convolution in the Colombeau algebra of generalized functions. One of reasons is that the Fourier transform of the infinite Radon measure $\delta_{\Lambda}$ is not defined for the measure $\delta_{\Lambda}$ but for the distribution $\delta_{\Lambda}$ (see Lemma \ref{lem:Lambda uniformly discrete}). Another is that $|\widehat{\mu}|^2$ is not necessarily integrable and thus the Fourier transform of $|\widehat{\mu}|^2$ may not defined for the function $|\widehat{\mu}|^2$ but for the distribution $|\widehat{\mu}|^2$.

Without loss of generality, we assume $0\in \Lambda$. Taking Fourier transform of both sides of (\ref{F-equation}), we have
\begin{equation}\label{F-equation 2}
\widehat{|\widehat{\mu}|^2} \cdot \widehat{\delta}_{\Lambda} = \delta_0,
\end{equation}
which is the main object of study throughout this section.
\begin{lem}\label{lem:non-negative}
	The distribution $\widehat{|\widehat{\mu}|^2}$ is non-negative.
\end{lem}
\begin{proof}
	It is sufficient to show that 
	\begin{equation}\label{eq:non-negative 1}
	\langle \widehat{|\widehat{\mu}|^2}, 1_{B(x,p^{k})}\rangle \ge 0
	\end{equation}
	for any given $x\in \Qp$ and $k\in \Z$.
	By definition of Fourier transformation on distribution, we have
	\begin{equation}\label{eq:non-negative 2}
	\langle \widehat{|\widehat{\mu}|^2}, 1_{B(x,p^{k})}\rangle=p^{k}\langle {|\widehat{\mu}|^2}, \chi_x1_{B(0,p^{-k})}\rangle.
	\end{equation}
	We define $\mu_{-}\in \mathcal{P}(\Qp)$ by $\mu_{-}(\Omega)=\mu(-\Omega)$ for any Borel set $\Omega\subset \Qp$. By Proposition \ref{Conv-Mul}, we get
	\begin{equation}\label{eq:non-negative 3}
	|\widehat{\mu}|^2=\widehat{\mu*\mu_{-}}.
	\end{equation}
	It follows from \eqref{eq:non-negative 2} and the fact that $\mu*\mu_{-}\in \mathcal{P}(\Qp)$ that
	\begin{equation*}
	\langle \widehat{|\widehat{\mu}|^2}, 1_{B(x,p^{k})}\rangle=p^{k}\langle \widehat{\mu*\mu_{-}}, \chi_x1_{B(0,p^{-k})}\rangle=\langle \mu*\mu_{-}, 1_{B(x,p^{k})}\rangle,
	\end{equation*}
	which is actually equal to $\mu*\mu_{-}(B(x,p^{k}))$. Thus we obtain \eqref{eq:non-negative 1}. 
\end{proof}
We remark that in the proof of Lemma \ref{lem:non-negative}, it not only tells that the distribution $\widehat{|\widehat{\mu}|^2}$ is non-negative, but only shows that $\widehat{|\widehat{\mu}|^2}=\mu*\mu_{-}$ in the sense of distribution.

Applying Lemma \ref{lem:non-negative} and Proposition \ref{prop:union of zeros} to the functional equation \eqref{F-equation 2}, we obtain
\begin{equation}\label{eq:spectral measure, zero}
\Qp\setminus\{0\}\subset \mathcal{Z}_{\widehat{|\widehat{\mu}|^2}} \cup \mathcal{Z}_{\widehat{\delta_{\Lambda}}}. 
\end{equation}
It roughly shows that the union of zeros of the distributions $\widehat{|\widehat{\mu}|^2}$ and $\widehat{\delta_{\Lambda}}$ are abundant. 
This is one of our motivation to investigate the measure $\mu$ and its spectrum $\Lambda$ by studying the zeros of distributions $\widehat{|\widehat{\mu}|^2}$ and $\widehat{\delta_{\Lambda}}$ in the following subsections.

\subsection{Structure of $\Lambda$}
In this section, we first justify that $\delta_{\Lambda}$ is indeed a distribution in $\Qp$ by the following lemma.
\begin{lem}\label{lem:Lambda uniformly discrete}
	The set $\Lambda$ is uniformly discrete.
\end{lem}
\begin{proof}
	Since $\mu \in \mathcal{P}(\Qp)$, it is well known that $\widehat{\mu}$ is continuous and $\widehat{\mu}(0)=1$. It follows that there exists $n \in \Z$ such that 
	\begin{equation}\label{eq:Lambda uniformly discrete}
	\widehat{\mu}(\xi)>\frac{1}{2}, ~\forall~|\xi|_p\le p^n.
	\end{equation}
	Since the set $\Lambda$ is the spectrum of $\mu$, we have $\widehat{\mu}(\lambda-\lambda')=0$ for any distinct $\lambda, \lambda'\in \Lambda$. Combining this with \eqref{eq:Lambda uniformly discrete}, we conclude that $|\lambda-\lambda'|_p>p^n$ for any distinct $\lambda, \lambda'\in \Lambda$.
\end{proof}

By Lemma \ref{zeroofE}, we know that if the point $x$ is contained in $\mathcal{Z}_{\widehat{\delta_{\Lambda}}}$, then the minimal sphere that is centered at $0$ and contains $x$ (that is, $S(0,p^{-n})$ with $|x|=p^{-n}$) is also included in $\mathcal{Z}_{\widehat{\delta_{\Lambda}}}$. 
Using this, in the following lemma, we analyze how the points in the set $\Lambda$ is distributed locally when a zero of the distribution $\delta_{\Lambda}$ is provided. Recall that $\sharp \Lambda_n^\xi=\sharp (\Lambda \cap B(\xi, p^n))$.
\begin{lem}\label{lem:zeros for delta_Lambda}
	Let $n\in \Z$. If $S(0,p^{-n}) \subset \mathcal{Z}_{\widehat{\delta_{\Lambda}}}$, then 
	\begin{equation}\label{eq:zeros for delta_Lambda 0}
	\sharp \Lambda_n^{\xi}=\sharp \Lambda_n^{\xi+p^{-n-1}},
	\end{equation}
	for every $\xi\in \Qp$.
\end{lem}
\begin{proof}
	Fix $n\in \Z$. Clearly, if $|\xi|_p>p^{-n-1}$ then $\Lambda_n^{\xi}= \Lambda_n^{\xi+p^{-n-1}}$ and \eqref{eq:zeros for delta_Lambda 0} follows. Thus it remains to consider the case where $|\xi|_p\le p^{-n-1}$.
	Fix $k\ge n+1$ and $\xi\in \Qp$ with $|\xi|_p=p^{-k}$.
	Since $S(0,p^{-n}) \subset \mathcal{Z}_{\widehat{\delta_{\Lambda}}}$ and $B(p^n, p^{-k})\subset S(0, p^{-n})$, we have that
	\begin{equation*}
	\langle \widehat{\delta}_{\Lambda}, 1_{B(p^n, p^{-k})} \rangle=0.
	\end{equation*}
	By the definition of Fourier transformation, we get 
	\begin{equation*}
	\langle \delta_{\Lambda}, \chi_{p^n} 1_{B(0, p^{k})} \rangle=0.
	\end{equation*}
	It follows that
	\begin{equation*}
	\sum_{\lambda\in \Lambda_{k}} \chi(p^n\lambda)=0.
	\end{equation*}
	By Lemma \ref{lem:p^n-cycles}, the set $\Lambda_{k}$ is a union of $p^{n+1}$-cycles.
	Observe that any $p^{n+1}$-cycle in $\Lambda_{k}$ either has exactly one element in each balls $B(\xi+jp^{-n-1}, p^{n})$ for $0\le j\le p-1$ or does not intersect  $\bigcup_{0\le j\le p-1}B(\xi+jp^{-n-1}, p^{n})$ at all. It follows that the number $\sharp \Lambda_n^{\xi+jp^{-n-1}}$ is a constant independent of $j$.
	We thus obtain \eqref{eq:zeros for delta_Lambda 0} by taking arbitrarily $\xi\in \Qp$ with $|\xi|_p\le p^{-n-1}$.
\end{proof}	

From Lemma \ref{lem:zeros for delta_Lambda}, we see that the distribution of $\Lambda$ is as in Figure \ref{Fig:spectrum} whenever $S(0,p^{-n}) \subset \mathcal{Z}_{\widehat{\delta_{\Lambda}}}$.

\begin{figure}[h!]
	\centering
	\includegraphics[width=0.8\textwidth]{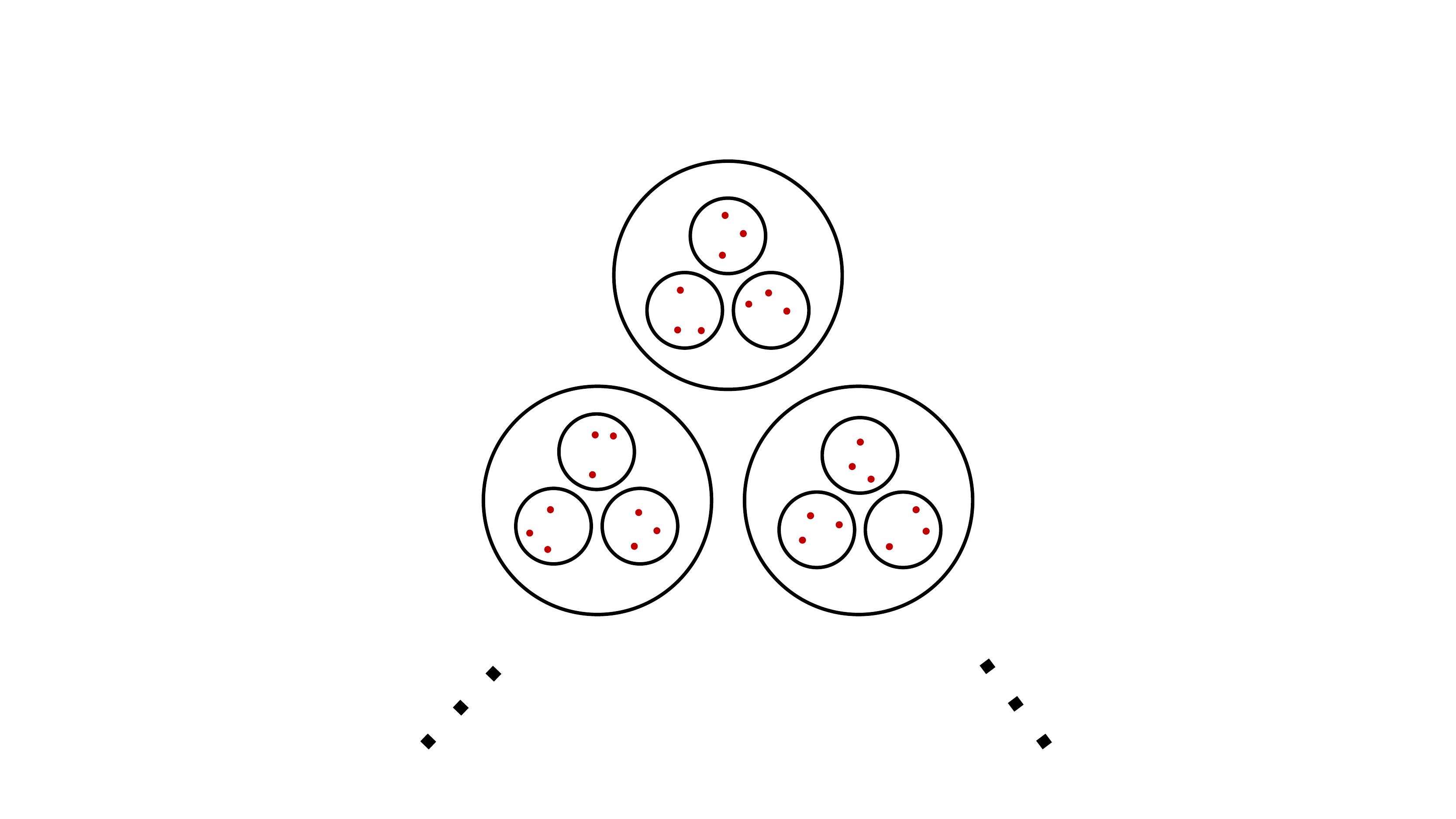}\\
	\caption{The distribution of $\Lambda$. The small ball has the radius $p^{n-1}$ and the big one $p^n$.} 
	\label{Fig:spectrum}
\end{figure}

Recall from Lemma  \ref{zeroofE} that $\mathcal{Z}_{\widehat{\delta_\Lambda}}$ has the following property: every sphere $S(0,p^{-n})$ either is contained in 
$\mathcal{Z}_{\widehat{\delta_\Lambda}}$ or does not intersect 
$\mathcal{Z}_{\widehat{\delta_\Lambda}}$.
Let
\begin{eqnarray*}	
	\mathbb{I}:&= &\left\{  n\in \Z:  S(0,p^{-n})\subset\mathcal{Z}_{\widehat{\delta_\Lambda}} \right\}, \\ 
	\mathbb{J}:&=&\{ n\in \Z: S(0,p^{-n})\cap \mathcal{Z}_{\widehat{\delta_\Lambda}}=\emptyset\},
\end{eqnarray*}
which form the partition of $\Z$. 
Due to Proposition \ref{zeroofE} and Lemma \ref{lem:Lambda uniformly discrete}, the set $\mathbb{J}$ contains the set $\{n\in \Z: n\le -n_{\Lambda}-2 \}$, where the integer $n_\Lambda$ is defined by \eqref{nE}.  Thus the set $\mathbb{I}$ has a minimal element, denoted by $i_\Lambda$, which is larger than $-n_\Lambda-1$. For $k\ge i_\Lambda$, let
$$
\mathbb{I}_{\le k}:= \left\{  n\in \mathbb{I}:  n\le k \right\},
$$
which is a finite set by the above statement. For $k\ge i_\Lambda$, let
$$
\mathbb{J}_{\le k}:= \left\{  n\in \mathbb{J}: i_\Lambda \le n\le k \right\},
$$
which is the complement set of $\mathbb{I}_{\le k}$ in $\{i_\Lambda, i_\Lambda+1, \dots, k \}$.

We will prove in the following lemma that every set $\Lambda_{k}^{\xi}$ contains a ``large" $p$-homogenous set.
\begin{lem}\label{lem:Lambda tree structure}
	For any $k\ge i_\Lambda$ and any $\xi\in \Lambda$,  the set $\Lambda_{k}^{\xi}$ contains a subset $C_k^\xi$ satisfying the condition that the set $C_k^\xi-\xi$ has $(k+1,k-i_\Lambda+1 ,I, J)$-tree structure with $I=k-\mathbb{I}_{\le k}$.
\end{lem}	
\begin{proof}
	We construct the set $C_k^{\xi}$ it by induction on $k$. When $k=i_\Lambda$, by Lemma \ref{lem:zeros for delta_Lambda}, we see that  $\Lambda_{i_\Lambda}^{\xi+jp^{-i_\Lambda-1}}$ is nonempty for any $\xi\in \Lambda$ and for every $0\le j\le p-1$ and thus define the set $C_{i_\Lambda}^\xi$ for every $\xi\in \Lambda$ by picking one element in each set $\Lambda_{i_\Lambda}^{\xi+jp^{-i_\Lambda-1}}$ for $0\le j\le p-1$. Clearly, we have $C_{i_\Lambda}^\xi\subset B(\xi, p^{i_\Lambda+1})$. It follows that 
	$$(p^{i_\Lambda+1}(C_{i_\Lambda}^\xi-\xi))_{\!\!\!\! \mod{p}}=\{0,1,\dots, p-1\}$$
	 which is a $\mathcal{T}_{\{0\}, \emptyset}$-form tree. This complete the proof for the case $k=i_\Lambda$ by noticing $\mathbb{I}_{\le i_\Lambda}=\{i_\Lambda\}$ and $\mathbb{J}_{\le i_\Lambda}=\emptyset$.
	
	Now assume that the set  $C_k^{\xi}$ are well defined for all $\xi\in \Qp$ and for all integer $n$ with $i_\Lambda\le n\le k$. We will construct $C_{k+1}^{\xi}$ for every $\xi\in \Qp$. If $k+1\notin \mathbb{I}$, then we see that $\mathbb{I}_{\le k+1}=\mathbb{I}_{\le k}$ and pick $C_{k+1}^\xi=C_k^\xi$ for every $\xi\in \Lambda$. Since $(k+1)-\mathbb{I}_{\le k+1}=(k-\mathbb{I}_{\le k})+1$, the set $C_{k+1}^\xi$ is desired by inductive hypothesis. Now suppose $k\in \mathbb{I}$. We pick $$C_{k+1}^\xi=\cup_{0\le j\le p-1} C_k^{\xi+jp^{-k-1}}$$ for every $\xi\in \Lambda$. Since $$C_k^{\xi+jp^{-k-1}}\subset \Lambda_{k}^{\xi+jp^{-k-1}}\subset B(\xi+jp^{-k-1}, p^k),$$ we see that $C_k^{\xi+jp^{-k-1}}$ are mutually disjoint for $0\le j\le p-1$. Moreover, the sets $(C_k^{\xi+jp^{-k-1}}-\xi)$  have  $(k+1,k-i_\Lambda+1 ,I, J)$-tree structure with $I=k-\mathbb{I}_{\le k}$. Thus we obtain that the set $(C_{k+1}^\xi-\xi)$ has  $(k+2,k-i_\Lambda+2 ,\{0\}\cup (I+1), J)$-tree structure. By observing that
	 $$\{0\}\cup (I+1)=\{0\}\cup((k+1)-\mathbb{I}_{\le k})=(k+1)-\mathbb{I}_{\le k+1},$$
	 we complete the proof.
\end{proof}	

A direct consequence of Lemma \ref{lem:Lambda tree structure} is the following.
\begin{cor}
	For any $k\in \Z$, the set $\Lambda_{k}$ contains a finite $p$-homogeneous set whose cardinality is $p^{\sharp \mathbb{I}_{\le k}}$. 
\end{cor}

\subsection{Structure of $\mu$}
We investigate the structure of $\mu$ by the help of zeros of the distribution $\widehat{|\widehat{\mu}|^2}$.
\begin{lem}\label{lem:zeros in mu hat hat hat}
	Let $n\in \Z$. If $S(0,p^{-n}) \subset \mathcal{Z}_{\widehat{|\widehat{\mu}|^2}}$, then 
	\begin{equation}
	\mu(B(\xi, p^{-k}))\mu(B(\xi-x, p^{-k}))=0,
	\end{equation}
	for any $\xi\in \Qp$, any $x\in S(0,p^{-n})$ and any $k\ge n+1$.
\end{lem}
\begin{proof}
	Fix $n\in \Z$. Fix $x\in S(0,p^{-n})$ and $k\ge n+1$. Since $S(0,p^{-n}) \subset \mathcal{Z}_{\widehat{|\widehat{\mu}|^2}}$, we have
	\begin{equation}\label{eq:zeros in mu hat hat hat 1}
	\langle \widehat{|\widehat{\mu}|^2}, 1_{B(x,p^{-k})}\rangle=0.
	\end{equation}
	By the remark after Lemma \ref{lem:non-negative}, we get from \eqref{eq:zeros in mu hat hat hat 1} that
	\begin{equation*}
	\langle \mu*\mu_{-}, 1_{B(x,p^{k})}\rangle=0.
	\end{equation*}
	This means that
	\begin{equation*}
	\int_{\Qp}\int_{\Qp} 1_{B(x,p^{-k})}(y+z) d\mu(y) d\mu(z)=0,
	\end{equation*}
	implying that
	\begin{equation}\label{eq:zeros in mu hat hat hat 5}
	\int_{\Qp} \mu(B(x-z, p^{-k})) d\mu(z)=0.
	\end{equation}
	We notice the fact that if $|z_1-z_2|\le p^{-k}$, then $B(x-z_1, p^{-k})=B(x-z_2, p^{-k})$. It follows that as a function of $z$, $\mu(B(x-z, p^{-k}))$ is locally constant.  Fix $\xi \in \Qp$. Let $P_{-k}$ be the set in $\Qp$ such that $\xi\in P_{-k}$ and $\{B(y, p^{-k}): y\in P_{-k} \}$ is a partition of $\Qp$. By \eqref{eq:zeros in mu hat hat hat 5}, we have
	\begin{equation}
	\sum_{y\in P_{-k}} \mu(B(x-y, p^{-k})) \mu(B(y, p^{-k}))=0.
	\end{equation}
	Since $\mu(B(x-y, p^{-k})) \mu(B(y, p^{-k}))\ge 0$, we obtain that 
	$$
	\mu(B(x-y, p^{-k})) \mu(B(y, p^{-k}))=0, \forall y\in P_{-k}.
	$$
	In particular, we have $\mu(B(x-\xi, p^{-k})) \mu(B(\xi, p^{-k}))=0$. This completes the proof.
\end{proof}	

Let
\begin{eqnarray*}	
	\mathbb{K}:= \left\{  n\in \Z:  S(0,p^{-n})\subset\mathcal{Z}_{\widehat{|\widehat{\mu}|^2}} \right\}.
\end{eqnarray*}
By \eqref{eq:spectral measure, zero} and the definition of $\mathbb{J}$, we have $\mathbb{J}\subset \mathbb{K}$. In particular, we have 
\begin{equation}\label{eq:zeros of K}
\{ n\in \mathbb{Z}: n\le i_\Lambda-1 \}\subset \mathbb{K}.
\end{equation}

\begin{prop}\label{prop:compact support}
	The spectral measure $\mu$ has compact support.
\end{prop}
\begin{proof}
	Without Loss of generality, we assume that the point $0$ is a density point of $\mu$, that is to say, $\mu(B(0, p^{n}))>0$ for all $n\in \Z$. We claim that the measure $\mu$ is supported on the ball $B(0, p^{-i_\Lambda})$. We prove our claim by contradiction. Assume that $\mu(B(\xi, p^{-k}))>0$ for some ball $B(\xi, p^{-k})$ which does not intersect the ball $B(0, p^{-i_\Lambda})$. Let $n_0:=v_p(|\xi|_p)$. Since  $B(\xi, p^{-k})\cap B(0, p^{-i_\Lambda})=\emptyset$, we obtain that $k\ge n_0+1$ and $i_\Lambda\ge n_0+1$. It follows that 
	\begin{equation*}
	B(\xi, p^{-k})\subset B(\xi, p^{-n_0-1})~\text{and}~ B(0, p^{-i_\Lambda})\subset B(0, p^{-n_0-1}).
	\end{equation*}
	Then we have 
	\begin{equation}\label{eq:compact support}
	\mu(B(\xi, p^{-n_0-1}))>0~\text{and}~\mu(B(0, p^{-n_0-1}))>0.
	\end{equation}
	On the other hand, by Lemma \ref{lem:zeros in mu hat hat hat} and \eqref{eq:zeros of K}, we see that
	\begin{equation*}
	\mu(B(0, p^{-n_0-1}))\mu(B(\xi, p^{-n_0-1}))=0,
	\end{equation*}
	which is contradict to \eqref{eq:compact support}. This completes the proof of our claim.
\end{proof}	

For $k\ge i_\Lambda$, let
$$
\mathbb{K}_{\le k}:= \left\{  n\in \mathbb{K}: i_\Lambda \le n\le k \right\}.
$$
 It is easy to see that 
\begin{equation}
\mathbb{J}_{\le k}\subset \mathbb{K}_{\le k}
\end{equation}
and
\begin{equation}\label{eq:relation I,K}
\mathbb{I}_{\le k}\cup \mathbb{K}_{\le k}=\{i_\Lambda, i_\Lambda+1, \dots, k\}.
\end{equation}

Obviously, the translation of $\mu$ does not change the spectrality of $\mu$.
In what follows, we always assume that the point $0$ is a density point of $\mu$. By the proof of Proposition \ref{prop:compact support}, we see that the measure $\mu$ is supported on the ball $B(0, p^{-i_\Lambda})$. Thus for any $k\ge i_\Lambda$ if $\mu(B(\xi, p^{-k-1}))\not=0$, then $\xi\in B(0, p^{-i_\Lambda})$. We observe that one way to represent the set of the centre of balls with  radius $p^{-k-1}$ in $B(0, p^{-i_\Lambda})$ is by the set
$$
p^{i_\Lambda}(\Z/p^{k-i_\Lambda+1}\Z):=\{0, p^{i_\Lambda}, 2p^{i_\Lambda}, \dots, (p^{k-i_\Lambda+1}-1)p^{i_\Lambda} \}.
$$
Let
$$\Omega_k:=\{x\in p^{i_\Lambda}(\Z/p^{k-i_\Lambda+1}\Z): \mu(B(x, p^{-k-1}))\not=0 \}.$$

We will prove in the following lemma that the set $\Omega_k$ is contained in a ``small" $p$-homogenous set.
\begin{lem}\label{lem:mu tree structure}
	For any $k\ge i_\Lambda$, 	the set $\Omega_k$ is contained in a $p$-homogeneous set $D_k$ having $(-i_\Lambda,k-i_\Lambda+1 ,I, J)$-tree structure with $J=\mathbb{K}_{\le k}-i_\Lambda$.
\end{lem}
\begin{proof}
	We construct the set $D_k$ it by induction on $k$. When $k=i_\Lambda$, we pick $D_{i_\Lambda}=p^{i_\Lambda}(\Z/p\Z)$. By the facts that $\sharp( p^{i_\Lambda}(\Z/p\Z))=p$ and that $\mathbb{K}_{\le i_\Lambda}=\emptyset$, we see that $D_{i_\Lambda}$ is what we desire.
	Now assume that the sets $D_n$ are well defined  for all integer $n$ with $i_\Lambda\le n\le k$. We observe that the balls $B(x+jp^{k+1}, p^{-(k+1)-1})$ ($0\le j\le p-1$) form a partition of the ball $B(x, p^{-k-1})$ for every $x\in p^{i_\Lambda}(\Z/p^{k-i_\Lambda+1}\Z)$. We first consider the case when $k+1\notin \mathbb{K}$. In such case, we have that $\mathbb{K}_{\le k+1}-i_\Lambda=\mathbb{K}_{\le k}-i_\Lambda$. Since $\Omega_k\subset D_k$ and $\Omega_{k+1}\subset \Omega_k+p^{k+1}(\Z/p\Z)$, we see that $\Omega_{k+1}\subset D_{k}+p^{k+1}(\Z/p\Z)$. Moreover, since the set $D_k$ has  $(-i_\Lambda,k-i_\Lambda+1 ,I, J)$-tree structure, it is not hard to check that  $D_k+p^{k+1}(\Z/p\Z)$ has  $(-i_\Lambda,k-i_\Lambda+2 ,I\cup\{k-i_\Lambda+1 \}, J)$-tree structure.
	Then we define $D_{k+1}=D_k+p^{k+1}(\Z/p\Z)$, which is desired by the above demonstration. Now suppose $k+1\in \mathbb{K}$. In such case, we claim that if $\mu(B(x, p^{-k-1}))\not=0$ for some $x\in p^{i_\Lambda}(\Z/p^{k-i_\Lambda+1}\Z)$, then there is exactly one ball among the balls $B(x+jp^{k+1}, p^{-(k+1)-1})$ for $0\le j\le p-1$, which doesn't have zero $\mu$-measure. This induces the embedding $\phi_k:\Omega_k\to p^{i_\Lambda}(\Z/p^{k-i_\Lambda+1}\Z)$ with $\phi(\Omega_{k})=\Omega_{k+1}$. We extend the domain of $\phi_k$ from  $\Omega_k$ to $D_k$ as follows: for $y\in D_k\setminus \Omega_k$, let $\phi_k(y)=y$. By the fact that $\Omega_k\subset D_k$, we have $\Omega_{k+1}\subset \phi_k(D_k)$. Moreover, since the set $D_k$ has  $(-i_\Lambda,k-i_\Lambda+1 ,I, J)$-tree structure, it is not hard to see that  the set $\phi_k(D_k)$ has  $(-i_\Lambda,k-i_\Lambda+2 ,I, J\cup \{k-i_\Lambda+1 \})$-tree structure. Let $D_{k+1}=\phi_k(D_k)$. Then we deduce that the set $D_{k+1}$ is desired by the above demonstration and the fact that $\mathbb{K}_{\le k+1}-i_\Lambda=(\mathbb{K}_{\le k}-i_\Lambda)\cup \{k-i_\Lambda+1 \}.$
	It remains to prove our claim. Assume that there exists $x\in p^{i_\Lambda}(\Z/p^{k-i_\Lambda+1}\Z)$ and distinct $0\le j,l\le p-1$ such that 
	\begin{equation*}
	\mu(B(x+jp^{k+1}, p^{-(k+1)-1}))\not=0~\text{and}~\mu(B(x+lp^{k+1}, p^{-(k+1)-1}))\not=0.
	\end{equation*}
	However by Lemma \ref{lem:zeros in mu hat hat hat}, we have that
	$$
	\mu(B(x+jp^{k+1}, p^{-(k+1)-1}))\mu(B(x+lp^{k+1}, p^{-(k+1)-1}))=0,
	$$
	which is impossible. This complete the proof of our claim.
\end{proof}	

The following is a direct consequence of Lemma \ref{lem:mu tree structure}.
\begin{cor}\label{prop:structure mu}
For all $k\ge i_\Lambda$, 	we have
\begin{equation}\label{eq:structure mu 0}
\sharp \{x\in p^{i_\Lambda}(\Z/p^{k-i_\Lambda+1}\Z): \mu(B(x, p^{-k-1}))\not=0 \}\le p^{k-i_\Lambda+1-\sharp \mathbb{K}_{\le k}}.
\end{equation}
\end{cor}

\subsection{Proof of Main theorem}
We first summarize what we have obtained in the previous sections. Recall that $\mu$ is the spectral measure with spectrum $\Lambda$. Since the translation  of $\mu$ is also a spectral measure, we might assume that $0\in \Lambda$ and  $0$ is a density point of $\mu$. 
Recall that
$\mathbb{I}$ is defined as the set consisting of the integer $n$ so that $S(0,p^n)$ is contained in the zero set of $\widehat{\delta_{\Lambda}}$. The set $\mathbb{K}$  is defined in the same way for $\widehat{|\widehat{\mu}|^2}$. The set $\mathbb{J}$ is the complement of $\mathbb{I}$ in $\Z$, in other words, it consists of the integer $n$ so that $S(0,p^n)$ is not contained in the zero set of $\widehat{\delta_{\Lambda}}$ (by Lemma \ref{lem:zeros for delta_Lambda}). By \eqref{F-equation 2}, the sets $\mathbb{I}, \mathbb{K}$ and $\mathbb{J}$ have the relation that
$$
\mathbb{I}\sqcup \mathbb{J}=\mathbb{Z}~\text{and}~ \mathbb{J} \subset \mathbb{K}.
$$
Recall that $\Lambda_k$ is the discrete set $\Lambda \cap B(0, p^k)$ and $\Omega_{k}$ consisting of the point $x$ with $\mu(B(x, p^{-k-1}))\not=0$. The sets $\Lambda_k$ and $\Omega_{k}$ are served as the local parts of $\Lambda$ and $\mu$ respectively.
We have shown that $\Lambda_k$ contains the ``large" $p$-homogeneous set $C_k$ $(=C_k^0)$ which has $(k+1,k-i_\Lambda+1 ,I, J)$-tree structure with $I=k-\mathbb{I}_{\le k}$ (Lemma \ref{lem:Lambda tree structure}). On the other hand, it has been shown that $\Omega_{k}$ is contained in the ``small" $p$-homogeneous set $D_k$ which has $(-i_\Lambda,k-i_\Lambda+1 ,I, J)$-tree structure with $J=\mathbb{K}_{\le k}-i_\Lambda$ (Lemma \ref{lem:mu tree structure}).

The following proposition is essential to prove Theorem \ref{thm:main}. We will show that $\mathbb{J}$ is actually equal to $\mathbb{K}$, and $D_k, C_k$ are of same size and have the ``complementary" tree structure.

\begin{prop}\label{prop:main theorem}
For any $k\ge i_\Lambda$, we have the following properties. 
\begin{itemize}
	\item [(1)] The value $\mu(B(x, p^{-k-1}))$ is a constant which is independent of  $x\in \Omega_k$.
	\item [(2)] $\Omega_k=D_k$.
	\item [(3)] $\mathbb{J}_{\le k}=\mathbb{K}_{\le k}$ and $\mathbb{I}_{\le k} \cap \mathbb{K}_{\le k}=\emptyset$.
	\item [(4)] $\Lambda_{k}=C_k$.
\end{itemize}

\end{prop}
\begin{proof}
By \eqref{eq:relation I,K} and Lemma \ref{lem:mu tree structure}, we obtain that for any $k\ge i_\Lambda$, there exists a $p$-homogeneous set $\widetilde{D}_k$ satisfying
\begin{itemize}
	\item $D_k\subset \widetilde{D}_k\subset p^{i_\Lambda}(\Z/p^{k-i_\Lambda+1}\Z)$;
	\item the set $\widetilde{D}_k$ has $(-i_\Lambda,k-i_\Lambda+1 ,I, J)$-tree structure with $I=\mathbb{I}_{\le k}-i_\Lambda$.
\end{itemize}

Since the set $\Lambda$ is the spectrum of $\mu$ and $0\in \Lambda$, we have $\widehat{\mu}(\xi)=0$ for all $\xi\in \Lambda\setminus\{0\}$. Fix $k\ge i_\Lambda$ and $\xi\in \Lambda_{k}\setminus\{0\}$. In particular,  we have
\begin{equation}\label{eq:main thm 1}
\int_{\Qp} \overline{\chi(\xi y)} d\mu(y)=0.
\end{equation}
Since $\chi(\xi \cdot)$ is locally constant, we compute that 
\begin{equation}\label{eq:main thm 2}
\begin{split}
\int_{\Qp} \overline{\chi(\xi y)} d\mu(y)&=\sum_{x\in p^{i_\Lambda}(\Z/p^{k-i_\Lambda+1}\Z)} \overline{\chi(\xi x)} \mu(B(x, p^{-k-1}))\\
&=\sum_{x\in \Omega_k} \overline{\chi(\xi x)} \mu(B(x, p^{-k-1})).
\end{split}
\end{equation}
Combining \eqref{eq:main thm 1} and \eqref{eq:main thm 2}, we have
\begin{equation}\label{eq:main thm 3}
\sum_{x\in \Omega_k} \overline{\chi(\xi x)} \mu(B(x, p^{-k-1}))=0.
\end{equation}
By arbitrary of $\xi\in \Lambda_{k}\setminus\{0\}$, the equation \eqref{eq:main thm 3} holds for every $\xi\in \Lambda_{k}\setminus\{0\}$. Let 
$$
v_k=(\mu(B(x, p^{-k-1})))_{x\in \Omega_k}
$$
be the vector in $\mathbb{R}^{\sharp \Omega_k}$.
Let
$$
M_k=\left( \overline{\chi(cx)}\right)_{c\in C_k\setminus\{0\}, x\in \Omega_k}
$$
be the complex $(\sharp C_k-1)\times \sharp \Omega_k$ matrix.
Since $C_k\subset \Lambda_{k}$, it follows from \eqref{eq:main thm 3} that
\begin{equation}\label{eq:main thm 4}
M_kv_k^T=0,
\end{equation}
where $\star^T$ stands for the transpose of $\star$. 

Now we calculate the rank of the matrix $M_k$. By Lemma \ref{lem:p homo spectral}, the matrix 
\begin{equation}
H_k=\left( \overline{\chi(cd)}\right)_{c\in C_k, d\in \widetilde{D}_k}
\end{equation} 
is a complex Hadamard matrix. In particular, the matrix $H_k$ has full rank. 
Clearly, the matrix $M_k$ is the submatrix of $H_k$, which is obtained by deleting one row and $\sharp (\widetilde{D}_k\setminus \Omega_k)$ columns of $H_k$. We denote by $H_k^\star$ the matrix that is obtained by deleting one row indexed by $0\in C_k$ of $H_k$. For $1\le j\le \sharp \widetilde{D}_k$, let $u_j$ be the $j$-th column of $H_k^\star$. By Lemma \ref{lem:p^n-cycles}, we have 
\begin{equation}\label{eq:main thm 5}
\sum_{j=1}^{\sharp{\widetilde{D}_k}}u_j=0.
\end{equation}
Since $H_k$ has full rank $\sharp\widetilde{D}_k$, we get that the rank of $\{u_j \}_{1\le j\le \sharp \widetilde{D}_k}$ is $\sharp\widetilde{D}_k-1$.  We claim that for any $1\le j_1<j_2<\dots<j_{ \sharp\widetilde{D}_k-1}\le \sharp \widetilde{D}_k$, the family $\{u_{j_\ell} \}_{1\le \ell\le \sharp\widetilde{D}_k-1}$ is linearly independent. In fact, if there exists $1\le j_1<j_2<\dots<j_{ \sharp\widetilde{D}_k-1}\le \sharp \widetilde{D}_k$ and $a_\ell\in \R$ not all equal zero for $1\le \ell\le \sharp\widetilde{D}_k-1$ such that $\sum_{\ell=1}^{\sharp{\widetilde{D}_k-1}} a_{j_\ell} u_{j_\ell}=0$. Combining this with \eqref{eq:main thm 5}, we obtain that the dimension of solution space $\{w: H_k^\star w=0 \}$ is large than or equal to $2$, which implies that the rank of $\{u_j \}_{1\le j\le \sharp \widetilde{D}_k}$ is smaller than $\sharp\widetilde{D}_k-1$. This is a contradiction. Therefore, we obtain that for any proper subset $T$ of $\widetilde{D}_k$, the rank of $\{u_j \}_{j\in T}$ is $\sharp T$. Consequently, the rank of the matrix $M_k$ is $\sharp \Omega_k$ if $\Omega_k$ is a proper subset of $\widetilde{D}_k$ and is $\sharp\Omega_{k}-1$ if $\Omega_{k}=\widetilde{D}_k$. Since the vector $v_k$ is nonzero, by \eqref{eq:main thm 4}, we obtain that the rank of the matrix $M_k$ has to be smaller than $\sharp\Omega_{k}$. Therefore we conclude that the the rank of the matrix $M_k$ is $\sharp\Omega_{k}-1$ and consequently that
$
\Omega_k=\widetilde{D}_k,
$
which implies (2).
Moreover, we have $M_k=H_k^\star$. The statement (3) follows from (2) and the simple fact that $I$ and $J$ don't intersect.

By \eqref{eq:main thm 5}, the solution space of $\{u:M_ku=0 \}$ is generated by the vector $(1,1,\dots, 1)^T$. Since $v_k$ is in this solution space, we conclude that
\begin{equation*}
\mu(B(x, p^{-k-1}))=\mu(B(y, p^{-k-1})), ~\forall~x,y\in \Omega_k,
\end{equation*}
which completes the proof of (1). 

It remains to prove (4). Due to (1) and (2), we have 
$$
\sum_{x\in \Omega_k} \chi((\xi-\xi')x)=0, ~\forall \xi\not=\xi'\in \Lambda_{k}.
$$
It follows that the family $\{\chi_{\xi} \}_{\xi\in \Lambda_{k}}$ is an orthogonal set of $L^2(\Omega_k)$. This implies $\sharp \Lambda_{k}\le \sharp \Omega_k$. Since $C_k\subset \Lambda_{k}$ and $\sharp C_k=\sharp \Omega_k$, we conclude that $\Lambda_{k}=C_k$. This completes the proof.

\end{proof}

In fact, as a consequence of Proposition \ref{prop:main theorem}, we could furthermore analyze the structure of the spectrum, that is,
\begin{equation}\label{eq:Lambda=Ck}
\Lambda=\cup_k C_k.
\end{equation}
Now we prove our main theorem.

\begin{proof}[Proof of Theorem \ref{thm:main}]
By Proposition \ref{prop:compact support}, we might assume that $\mu$ is supported on $\Zp$ and $i_\Lambda=0$. By Proposition \ref{prop:main theorem} (1), we obtain that $\mu$ is the weak limit of 
$$
\frac{1}{\sharp \Omega_k}\sum_{x\in \Omega_k} \delta_x,
$$
as $k$ tends to $+\infty$.
By Proposition \ref{prop:main theorem} (2) and (3), we have
\begin{equation}\label{eq:main thm 01}
\frac{1}{\sharp \Omega_k}\sum_{x\in \Omega_k} \delta_x=\frac{1}{\sharp \mathbb{I}_{\le k}}\sum_{x\in C_{\mathbb{I}_{\le k}, \mathbb{J}_{\le k}}} \delta_x,
\end{equation}
for all $k\ge 0$. Observe that the RHS of \eqref{eq:main thm 01} weakly converges to $\nu_{\mathbb{I},\mathbb{J}}$ as $k$ tends to $+\infty$. This completes the proof.
\end{proof}

\section{Dimension of spectral measures}\label{sec:cor dimension}

In this section, we prove Proposition \ref{cor:dimension}.
\begin{proof}[Proof of Proposition \ref{cor:dimension}] 
	It is sufficient to prove for the measure $\mu=\nu_{\mathbb{I},\mathbb{J}}$ and $\Lambda$. By calculation, we have
	\begin{equation}\label{eq:cor 1}
\underline{\dim}_e \mu=\liminf\limits_{k\to \infty} \frac{\sharp \mathbb{I}_{\le k}}{k}~\text{and}~
\overline{\dim}_{\text{e}} \mu=\limsup\limits_{k\to \infty} \frac{\sharp \mathbb{I}_{\le k}}{k}.
	\end{equation}
	A simple computation shows that for all $x\in \text{supp}(\mu)$, we have 
	\begin{equation*}
	\underline{d}(\mu, x)=\liminf\limits_{k\to \infty} \frac{\log \mu(B(x,p^{-k}))}{k}=\liminf\limits_{k\to \infty} \frac{\sharp \mathbb{I}_{\le k}}{k}.
	\end{equation*}
	By the fact \eqref{eq:dimension fact}, we have 
	\begin{equation}\label{eq:cor 2}
	\dim_H \mu=\liminf\limits_{k\to \infty} \frac{\sharp \mathbb{I}_{\le k}}{k}.
	\end{equation}
 On the other hand, by \eqref{eq:Lambda=Ck}, we have
	\begin{equation}\label{eq:cor 3}
	\dim_B \Lambda = \limsup\limits_{k\to \infty} \frac{\sharp \mathbb{I}_{\le k}}{k}.
	\end{equation}
	We complete the proof by combining \eqref{eq:cor 1}, \eqref{eq:cor 2} and \eqref{eq:cor 3}.
\end{proof}

\section{Higher dimensional spectral measures}\label{sec:higher dimension}

In this section, we show some properties of spectral measures in $\Qp^d$ and discuss several differences between spectral measures in $\Qp$ and ones in $\Qp^d$, $d\ge 2$.
\subsection{Pure type  phenomenon}
As an analogy of the pure type  phenomenon  of spectral measures in $\R^n$ \cite{HeLaiLau2013}, we have the following proposition. The proof of the first part is similar to the Euclidean case and the second part is due to Proposition \ref{prop:density zero} and \cite[Theorem 3.1 (3)]{FFLS}.  We omit the proof and leave the readers to work out the details.
\begin{prop}\label{lem:pure type phenomenon}
	 Let $\mu\in \mathcal{M}(\Qp^d)$ be a spectral measure with spectrum $\Lambda$.  Then it must be one of the three pure types: discrete (and finite), singularly continuous or absolutely continuous. Moreover, the the following holds.
	\begin{itemize}
		\item [(1)] If $\mu$ is discrete, namely $\delta_C$, then $\sharp C<\infty$ and $\sharp \Lambda<\infty$;
		\item [(2)] If $\mu$ is singularly continuous, then $D(\Lambda)=0$.
		\item [(3)] If $\mu$ is absolutely continuous, then $\mu(\Qp^d)<\infty$ and $D(\Lambda)=1/\mu(\Qp^d)$.
	\end{itemize}
\end{prop}
In fact, the same method works for general locally compact abelian groups.

\subsection{Spectral measures in $\Q_p^d, d\ge 2$}
For a uniformly discrete set $E$ in the higher dimensional space $\Q_p^d$ with $d\geq 2$, the zero set of the Fourier transform of the measure $\delta_E$ is not necessarily  bounded.  In other words,
Proposition \ref{zeroofE} does not hold in $\Q_p^d$ with $d\geq 2$.   For example,  let $E=\{(0,0),(0,1),\cdots,(0,p-1) \}$ which is a finite subset of $\Q_p^2$. One can check that 
$$\Z_{\widehat{\delta_E}}= \Q_p \times p^{-1} \Z_p^{\times},
$$
which is unbounded. In \cite{FFLS}, a spectral set in $\Qp^2$ which is not compact open was constructed: we partition $\Z_p$ into $p$ Borel sets of same Haar measure, denoted $A_i$ $(0 \le i\le p-1)$, assume that one of $A_i$ is not compact open, and define $$
\Omega:= \bigcup_{i= 0}^{p-1}  A_i \times B(i, p^{-1})\subset \Z_p\times \Z_p,
$$
which is a spectral set and not compact open in $\Zp\times \Z_p$.
 We remark that such example shows the measure $1_\Omega dx$ is the spectral measure in $\Q_p^2$ but it is not a translation or  multiplier of $\nu_{\mathbb{I},\mathbb{J}}$, where $\mathbb{I}$ and $\mathbb{J}$ form a partition of $\mathbb{N}^2$.


\subsection{Dimensions of spectra}
By using the same method in \cite{Shi3} where the author investigated the dimension of spectra of spectral measures in $\R^d$, we could prove the following proposition.
\begin{prop}\label{prop:entropy dimension}
	Let $\mu$ be a spectral measure in $\Qp^d$ with spectrum $\Lambda$. Then we have
	$$
	\dim_B \Lambda\le\overline{\dim}_{\text{e}} \mu.
	$$
\end{prop}
Regarding to Proposition \ref{cor:dimension}, we might ask the question whether the equality in Proposition \ref{prop:entropy dimension} still holds  for $d\ge 2$. In fact, if the measure $\mu$ is absolutely continuous or discrete, then the equality in Proposition \ref{prop:entropy dimension} holds trivially. Therefore, the question is only asked for singular continuous spectral measures in $\Qp^d$ for $d\ge 2$. Unfortunately, we could not answer it now.



\begin{thebibliography}{ds}	


\bibitem{Aks}
	S. Albeverio, A. Khrennikov and V. Shelkovich,
	{\em Theory of $p$-adic distributions: linear and nonlinear models}.
	Oxford University Press, Oxford, 2010.
	
	\bibitem{Afl}
	L. An, X Fu, CK Lai, {\em On Spectral Cantor-Moran measures and a variant of Bourgain's sum of sine problem,} Adv in Math, 349 (2019), 84-124.
	
	
	\bibitem{Fan}
	A. H. Fan, {\em Spectral measures on local fields}. pp. 15-35,
	in Difference Equations, Discrete Dynamical Systems
	and Applications, M. Bohner et al. (eds.), Springer Proceedings in Mathematics \& Statistics 150.
	Springer International Publishing Switzerland 2015.
	  arXiv:1505.06230
	
	
	\bibitem{FFS}
	A. H. Fan, S. L. Fan and  R. X. Shi, {\em Compact open spectral sets in $\mathbb{Q}_p$,} J. Funct. Anal. 271 (2016), no. 12, 3628-3661.
	
	\bibitem{FFLS}
	 A. H. Fan, S. L. Fan, L. M. Liao and  R. X. Shi, {\em Fuglede's conjecture holds in $\mathbb{Q}_p$, } Mathematische Annalen, 2019, 375(1-2): 315-341.
	\bibitem{FG}
	B. Farkas and R. S. Gy, {\em Tiles with no spectra in dimension 4}. Mathematica Scandinavica, 98 (2006), 44-52.
	
\bibitem{FMM06}
B. Farkas, M. Matolcsi and   P. M{\'o}ra,
  {\em On Fuglede's conjecture and the existence of universal spectra. }
J. Fourier Anal. Appl. 12 (2006), no. 5, 483-494. 
	
	\bibitem{F}
	B. Fuglede,
	{\em Commuting self-adjoint partial differential operators and a group theoretic problem}. J. Funct. Anal., 16 (1974), 101-121.
	

\bibitem{HeLaiLau2013}
X. G. He, C. K. Lai and K. S. Lau, {\em Exponential spectra in $L^2(\mu)$,} Appl.
Comput. Harmon. Anal. 34 (2013), no. 3, 327-338.


\bibitem{imp}
A. Iosevich, A. Mayeli, J. Pakianathan, {\em The Fuglede conjecture holds in $\mathbb{Z}_p \times \mathbb{Z}_p$,} Analysis \& PDE, 2017, 10(4): 757-764.
	

\bibitem{jp}	
 P. Jorgensen, S. Pedersen, {\em Dense analytic subspaces in fractal $L^2$ spaces,} J. Anal. Math. 75 (1998) 185-228.	

	


%

	


	\bibitem{KM}
	M. N. Kolountzakis and  M. Matolcsi, {\em Tiles with no spectra}, Forum Mathematicum, 18 (2006), 519-528.
	
	\bibitem{KM2}	M. N. Kolountzakis and M. Matolcsi, {\em Complex Hadamard matrices and the spectral set conjecture}, Collectanea Mathematica, 57 (2006), 281-291.

	
	\bibitem{kmsv}
	G. Kiss, R-D. Malikiosis, G. Somlai, M. Vizer, {\em On the discrete Fuglede and Pompeiu problems,} arXiv:1807.02844, 2018.
	
	\bibitem{M}	M. Matolcsi, {\em Fuglede conjecture fails in dimension 4}, Proceedings of the American Mathematical Society, 133 (2005), 3021-3026.
	
	\bibitem{l2}
	I.  \L aba, {\em The spectral set conjecture and multiplicative properties of roots of polynomials,} J. London Math. Soc. (2) 65 (2002), no. 3, 661-671.
	





	\bibitem{s}
	I. J. Schoenberg,
	{\em A note on the cyclotomic polynomial. }
	 Mathematika, (1964), 11(02): 131-136.
	
	
\bibitem{Shi1}
R. X. Shi,	{\em Fuglede's conjecture holds on cyclic groups $\mathbb{Z}_{pqr}$,} Discrete analysis, 2019:14, 14 pp.
	
\bibitem{Shi2}
R. X. Shi,	{\em Equi-distributed property and spectral set conjecture on $\mathbb{Z}_{p^2}\times \Z_{p}$,} arXiv:1906.11717, 2019.

\bibitem{Shi4}
R. Shi, {\em Spectrality of a class of CantorMoran measures,}  Journal of Functional Analysis, 2019, 276(12): 3767-3794.


\bibitem{Shi3}
R. X. Shi, {\em On dimensions of frame spectral measures and their frame spectra,} arXiv:2002.03855, 2020.
	
\bibitem{mk}
R-D. Malikiosis, M. N. Kolountzakis,
{\em Fuglede's conjecture on cyclic groups of order $p^nq$,} Discrete Analysis, 2017:12, 16pp.	
	
	

	\bibitem{t}
	M. H. Taibleson,
	{\em Fourier analysis on local fields.}
	 Princeton University Press, 1975.
	
	
	
	\bibitem{TT}
	T. Tao, {\em Fuglede's conjecture is false in 5 and higher dimensions}, Math.
	Research Letters, 11 (2004), 251-258.
	
	
	
	

   	\bibitem{Vvz}
   	V. S. Vladimirov, I. V. Volovich and E. I. Zelenov, {\em P-adic analysis and mathematical physics}, World Scientific Publishing, 1994.
   	
	
	
\end{thebibliography}
\end{document}